\documentclass[12pt,leqno]{amsart}
\usepackage{amssymb, graphicx, amscd, amsmath, amsthm, 
array, fancyvrb, relsize, paralist,url}
\usepackage{fullpage}
\usepackage{colonequals}
\usepackage[all]{xy}

\makeatletter
\g@addto@macro\bfseries{\boldmath}
\makeatother

\theoremstyle{plain}
\newtheorem{thm}{Theorem}[section]
\newtheorem{lem}[thm]{Lemma}
\newtheorem{cor}[thm]{Corollary}

\newtheorem{conj}[thm]{Conjecture}
\newtheorem*{exa}{Example}
\newtheorem*{rem}{Remark}

\def\Gal{\mathrm{Gal}}
\def\ord{\mathrm{ord}}
\def\bcases{\begin{cases}}
\def\ecases{\end{cases}}
\def\t{\hbox}
\def\f{\frac}
\def\e{\equiv}
\newcommand{\jac}[2]{\left(\displaystyle{\frac{#1}{#2}}\right)}

\newcommand{\Q}{\mathbb{Q}}
\newcommand{\Z}{\mathbb{Z}}
\newcommand{\OO}{\mathcal{O}}
\newcommand{\p}{\mathfrak{p}}
\newcommand{\q}{\mathfrak{q}}
\newcommand{\PP}{\mathfrak{P}}
\newcommand{\QQ}{\mathfrak{Q}}
\newcommand{\A}{\mathfrak{A}}
\renewcommand{\a}{\alpha}
\renewcommand{\b}{\beta}
\renewcommand{\c}{\gamma}

\numberwithin{equation}{section}

\begin{document}

\title{Ring class fields and a result of Hasse}

\author{Ron Evans}
\address{Department of Mathematics, UCSD\\
La Jolla, CA  92093-0112}
\email{revans@ucsd.edu}
\urladdr{https://mathweb.ucsd.edu/~revans}

\author{Franz Lemmermeyer}
\address{M\"{o}rikeweg 1\\
73489 Jagstzell, Germany}
\email{hb3@uni-heidelberg.de}
\urladdr{https://www.mathi.uni-heidelberg.de/~flemmermeyer}

\author{Zhi-Hong Sun}
\address{\parbox{\linewidth}{School of Mathematics and Statistics,\\
Huaiyin Normal University,\\
Huaian, Jiangsu 223300, P.R. China}}
\email{zhsun@hytc.edu.cn}
\urladdr{http://maths.hytc.edu.cn/szh1.htm}

\author{Mark Van Veen}
\address{2138 Edinburg Avenue\\
Cardiff by the Sea, CA 92007}
\email{mavanveen@ucsd.edu}
\urladdr{}



\date{April 2024}

\begin{abstract}
For squarefree $d>1$, let $M$ denote the ring class field for the order
$\Z[\sqrt{-3d}]$ in $F=\Q(\sqrt{-3d})$.
Hasse proved that $3$ divides the class
number of $F$ if and only if there exists
a cubic extension $E$ of $\Q$ such that $E$ and $F$
have the same discriminant.
Define the real cube roots $v=(a+b\sqrt{d})^{1/3}$
and $v'=(a-b\sqrt{d})^{1/3}$, where $a+b\sqrt{d}$
is the fundamental unit in $\Q(\sqrt{d})$.
We prove that $E$ can be taken as $\Q(v+v')$
if and only if $v \in M$.   As byproducts of the proof,
we give explicit congruences for $a$ and $b$
which hold if and only if $v \in M$, and we also show that
the norm of the relative discriminant of $F(v)/F$
lies in $\{1, 3^6\}$ or $\{3^8, 3^{18}\}$
according as $v \in M$ or $v \notin M$.
We then prove that $v$ is always in
the ring class field for the order
$\Z[\sqrt{-27d}]$ in $F$.
Some of the results above are extended for subsets of $\Q(\sqrt{d})$ properly
containing the fundamental units $a+b\sqrt{d}$. 
\end{abstract}

\maketitle

\section{Introduction}
Write $u=a + b\sqrt{d}$ for the fundamental
unit in $\Q(\sqrt{d})$, where $d>1$ is squarefree.
Define the real cube roots
$v=(a+b \sqrt {d})^{1/3}$ and $v'=(a-b \sqrt {d})^{1/3}$. 
Note that $v v'=\pm 1$.
Write $t=v+v'$ and
$F=\Q(\sqrt{-3d})$.
Let $M$ denote the ring class field for the order $\Z[\sqrt{-3d}]$ in $F$.
For number field extensions $K/k$, write
$D(K/k)$   for the relative discriminant
and $D(K)$ for the discriminant.

In \cite{H},  
Hasse proved that $3$ divides the class number of $F$  
if and only if  there exists a cubic extension $E$ of  $\Q$ such that
$D(E)=D(F)$.
When $v \in M$,  we prove in Theorems 1.1 and 1.2 below that
$E$ can be taken to be  $\Q(t)$.
Theorems 1.1 and 1.2 address the cases  $3 \nmid d$ and
$3 \mid d$, respectively.   They are proved in  Sections 2 and 3.
\begin{thm}\label{Thm 1.1}
Suppose that $3 \nmid d$ and $v \in M$.   Then the fields $E=\Q(t)$
and $F$ have the same discriminant.
\end{thm}
\begin{thm}\label{Thm 1.2}
Suppose that $3 \mid d$ and $v \in M$.   Then the fields $E=\Q(t)$
and $F$ have the same discriminant.
\end{thm}

Remarks (2A) and (3B) show that $v \in M$ implies that 
$3$ divides the class number of $F$.  However,
the converse is false; the smallest counterexample when $3 \nmid d$
is $d=142$, while the smallest counterexample when $3 \mid d$ is
$d=786$.  See \cite{HC, I,  KM, XC} for examples of infinite families of $d$ 
for which $3$ divides the class number of $F$.

When $v \notin M$, then in contrast with the theorems above, 
the discriminant of $\Q(t)$
does not equal $D(F)$.
In fact, when $v \notin M$,  
Theorems 1.3 and 1.4 show that $D(\Q(t))$
equals $9D(F)$ or $81D(F)$ according as $3 \nmid d$ or $3 \mid d$.
The proofs are given in Sections 4 and 5.
\begin{thm}\label{Thm 1.3}
Suppose that $3 \nmid d$ and $v \notin M$.   Then the field $E=\Q(t)$
has discriminant $9D(F)$.
\end{thm}
\begin{thm}\label{Thm 1.4}
Suppose that $3 \mid d$ and $v \notin M$.   Then the field $E=\Q(t)$
has discriminant $81D(F)$.
\end{thm}

The four proofs will use the following notation:
$w = (-1 + \sqrt{-3})/2$,
$C=\Q(\sqrt{d})$,  $B=\Q(\sqrt{-3},\sqrt{d})=F(\sqrt{d})=F(w)$,
$k=\Q(t)$, $K=\Q(v)=k(\sqrt{d})$, and
$L = F(v) = \Q(w,v) =K(w)$.
Note that $|L:B|=|K:C|=|k:\Q|=|F(t):F|=3$.
We remark in passing that by the Scholz reflection principle,
$3$ divides the class number of $F$
whenever $3$ divides the class number of $C$ \cite[Theorem 5]{CG}.
The Galois extension $L/F$ is cyclic of degree 6.
As a byproduct of the proofs, Theorem 1.5 gives an
evaluation of the norm of $D(L/F)$;  for a generalization, see Conjecture 7.7.
\begin{thm}\label{Thm 1.5}
When $v \in M$, the
norm of $D(L/F)$
equals $1$ or $3^6$ according as $3\nmid d$ or $3 \mid d$;
and when $v \notin M$, the norm of $D(L/F)$ equals $3^8$ or $3^{18}$
according as $3\nmid d$ or $3 \mid d$.  
\end{thm}
\begin{proof}
See Remarks
(2A), (3D), (4F), and (5G), respectively.
\end{proof}

The inclusion $v \in M$ is connected to cubic residuacity of 
$u = a + b \sqrt{d}  \pmod{p}$, 
where $p$ is any prime of the form $p=x^2 + 3dy^2$.
This is shown in Theorem 1.6 below.

When $d=2$, we have $\jac{d}{p}=1$.  When $d$ is odd,
one of $d,p$ is $1 \pmod{4}$, so that 
$\jac{d}{p}=\jac{p}{d} =1$.  Thus $d$ is a square $\pmod{p}$, so that
$u=a+b \sqrt{d}$ can be viewed as a rational integer $u_p \pmod{p}$.

\begin{thm}\label{Thm 1.6}
$v \in M$ if and only if 
$u_p$ is a cubic residue mod the primes $p=x^2 + 3dy^2$.
\end{thm}
\begin{proof}
Consider the principal prime ideal $\p=(x+y\sqrt{-3d})$ in $F$
of norm $p$.  By Theorem 1.5, $\p$ is unramified in $L$.
Let $\sigma$ denote the Artin symbol $\jac{L/F}{\p}$,
and let $\PP$ be a prime ideal in $L$ above $\p$.
Since $\sigma(v) \equiv v^p \pmod{\PP}$, we see that
$\sigma$ is trivial on $L$ if and only if 
$v^{p-1} \equiv 1 \pmod {\PP}$.
This last congruence is equivalent to 
$u_p^{(p-1)/3} \equiv 1 \pmod{p}$.
Thus $u_p$ is a cubic residue $\pmod{p}$ if and only if 
$\sigma$ is trivial on $L$.
It remains to show that $\sigma$ is trivial on $L$ if and only if
$L \subset M$.
We know that $\sigma$ is trivial on $L$ if and only if the primes $p$
split completely in $L$.  By \cite[Thm. 8.19]{Cox},
the primes $p$ split completely in $L$ if and only if
$L \subset M$, so the proof is complete.
\end{proof}

Let $M_c$ denote the ring class field of $F$ for the order
$\Z[\sqrt{-3dc^2}]$. Thus $M_1=M$ and $M_3$ is the
ring class field of $F$ for the order $\Z[\sqrt{-27d}]$.
Mimicking the proof of Theorem 1.6, we see that 
$v \in M_c$ if and only if
$u_p$ is a cubic residue mod the primes $p=x^2 + 3d(cy)^2$.

In Theorem 6.1, 
we give explicit criteria in terms of $a$ and $b$ for
$v=(a+b\sqrt{d})^{1/3}$ to lie in $M$,
where $a+b\sqrt{d}$ is the fundamental unit.
Theorem 6.2 shows that every $v$ lies in $M_3$.
Theorem 6.3 relates 
the ring class fields $M$ and $M_3$ to ray class fields of $F$.

In Section 7,  
we introduce a large class $S_d$ of integers $r+s\sqrt{d}$
with cubic norms, where $S_d$ properly contains the set of fundamental units
$a+b\sqrt{d}$.  Under certain conditions,
we extend Theorem 6.1 for elements in $S_d$.

A substantial generalization of Theorem 6.1 is given in
Section 8.
The proof makes no appeal to class field theory, but instead
relies wholly on the methods in \cite{S}.
As a corollary, we provide congruences for certain Lucas
numbers modulo primes $p=x^2 + 3dy^2$.

\section{Proof of Theorem 1.1}
Assume throughout this section that $v \in M$.
The proof of Theorem 1.1 utilizes the following five lemmas.  
\begin{lem}\label{Lemma 2.1}
When $3 \nmid d$, $L/F$ is unramified.
\end{lem}
\begin{proof}
By hypothesis, $L \subset M$. 
We may assume that $d \equiv 1 \pmod{8}$;
otherwise, by \cite[Thm. 7.24]{Cox},
$M$ is the Hilbert  class field of $F$ so that $L/F$
is unramified.  Under this assumption, 
the order $\Z[\sqrt{-3d}]$ has conductor $2$,
so it suffices to show that $2$ is unramified in $L$.
Consider the tower $\Q \subset C \subset K \subset L$.
Clearly $C/\Q$ is unramified at $2$.  
Also $K/C= \Q(v)/\Q(\sqrt{d})$ is unramified at $2$,
since the polynomial $x^3 - v^3$ has discriminant $-27v^6$.
Finally, $L/K= K(w)/K$ is unramified at $2$, since the polynomial
$x^2+x+1$ has discriminant $-3$.
\end{proof}

\begin{rem}[\bf{2A}]
Let $3 \nmid d$.
Since $L/F$ is a cyclic unramified extension of degree $6$
when $3 \nmid d$ by Lemma 2.1,
the class number of $F$ is divisible by $6$, and
$D(L/F)$ has (absolute) norm $1$.   This is in contrast with
the case $3 \mid d$; see Remark (3D).
\end{rem}

\begin{lem}\label{Lemma 2.2}
When $3 \nmid d$, $D(L/K)$ has norm $9$.
\end{lem}
\begin{proof}
We need only examine the ramification at $3$,
since the polynomial
$x^2+x+1$ has discriminant $-3$.
Note that $3$ ramifies in $F$ but there can be no further ramification
at $3$ in $L/F$ by Lemma 2.1.  Thus in the factorization of $(3)$
in $L$, every prime ideal occurs to the second power.

The minimal polynomial of $t=v+v'$ over $\Q$ is $x^3 -3 \epsilon x  -2a$,
where $\epsilon$ is the norm of the fundamental unit $v^3 = a+b\sqrt{d}$.
This polynomial has discriminant $-108 d b^2$, in which the exponent
of the factor $3$ is odd.  By \cite[Prop. 2.13]{N},
$3$ divides $D(k)$, so
$3$ must ramify in $k$ and in $K$.

The factorization of $(3)$ in $C$ is either
$\q$ or $\p \p'$,  where $\q$ has norm $9$
and $\p$, $\p'$ each have norm $3$.
Since $3$ ramifies in $K$, we have in $K$ either 
the prime ideal factorization $(\q) =\QQ_1^2 \QQ $ 
where the prime factors have norm $9$, or
$(\p \p') =\PP \PP_1^2 \PP' \PP_1'^2$ where the prime factors have norm $3$.
In the first case, $\QQ$ is the factor that ramifies in $L$,
while in the second case, $\PP$ and $\PP'$ are the factors that 
ramify in $L$.   Since the ramification is tame,
$D(L/K)$ equals $\QQ$ or $\PP \PP'$, and in either case
$D(L/K)$ has norm 9.
\end{proof}

\begin{lem}\label{Lemma 2.3}
When $q \ne 3$ is a rational prime, $(q) \nmid D(K/k)$ in $k$.
\end{lem}
\begin{proof}
Suppose for the purpose of contradiction that $(q)$ divides $D(K/k)$.
Since $q$ divides the discriminant of the $k$-basis $\{1,v \}$ in $K$,
$q$ must divide
\begin{equation}\label{2.1}
t^2-4 = v^2 + v'^2 -2.
\end{equation}
Replacing $v$ in (2.1) by its
conjugate $w v$ in $L$, we see that in $L$,  $q$ divides
\begin{equation}\label{2.2}
w^2 v^2 + w v'^2 -2.
\end{equation}
Subtracting, we see that $q$ divides
\begin{equation}\label{2.3}
(w^2-1) v^2 + (w-1) v'^2.
\end{equation}
Replacing $w$ by $w^2$ in (2.3), we see that $q$ divides
\begin{equation}\label{2.4}
(w-1) v^2 + (w^2-1) v'^2.
\end{equation}
Multiplying (2.3) by $w+1$ and then subtracting from (2.4),
we see that $q$ divides $3wv^2$, which is impossible, since $q \neq 3$.
This contradiction  proves the lemma.
\end{proof}

\begin{lem}\label{Lemma 2.4}
Suppose that $3 \nmid d$. 
The only rational primes that ramify in $L$ are the ones that divide
$D(F)$.  If the rational prime $q$ divides $D(C)$,
then each prime ideal in the factorizations of $(q)$
in $K$ and $L$ occurs to the second power.
\end{lem}
\begin{proof}
If the prime $p$ divides $D(F)$, then $p$ ramifies in $F$ with ramification
index $2$. Then since $L/F$ is unramified by Lemma 2.1,  
each prime ideal in the factorization of $(p)$
in $L$ occurs to the second power.
If $p$ does not divide $D(F)$, then clearly $p$ does not ramify in $L$.

Next, suppose that $q$ divides $D(C)$ (so that
$q \ne 3$).   Then $q$ ramifies in $C$ with ramification
index $2$. 
Recall that $K/C$ is unramified at $q$,
since the polynomial $x^3 - v^3$ has discriminant $-27v^6$.
Thus each prime ideal factor in the factorization of $(q)$
in $K$ occurs to the second power.
\end{proof}

\begin{lem}\label{Lemma 2.5}
When $3 \nmid d$, $D(K) = 9 D(C)^3$.
\end{lem}
\begin{proof}
By Lemma 2.1 and \cite[Prop. 4.15]{N},
$D(L)  = D(F)^6$.
Thus by Lemma 2.2,
\begin{equation}\label{2.5}
D(K)^2 = D(L)/9 = D(F)^6/9.
\end{equation}
Since $D(K)$ is positive \cite[Prop. 2.15]{N},
\begin{equation}\label{2.6}
D(K) = - D(F)^3/3 = 9D(C)^3.
\end{equation}
\end{proof}

We are now prepared for the proof of Theorem 1.1.

\begin{proof}
Consider the set  $S$ of rational primes $p$ that divide
the discriminant $D(C)$. Note that $3 \notin S$. By Lemma 2.4,
each prime ideal in the factorization of $(p)$ in $K$ (as
well as in $L$) occurs to the second power.
Each prime ideal in the factorization of $(p)$ in $k$
must occur to either
the first or second power, and those occurring to the first power are
exactly the ones that ramify in $K$.   Those that ramify tamely
in $K$ are the only ones that divide $D(K/k)$ to the first power
\cite[p. 260]{N}.

For $p \in S$, 
if every prime ideal factor of $(p)$ in $k$ were to occur to the
first power, then $(p)$ would divide $D(K/k)$, contradicting Lemma 2.3.
Thus $p$ ramifies in $k$, so there is a unique prime ideal $\p$ in $k$
that divides $(p)$ to the first power. Note that $\p$ has norm $p$.
If $p=2$, then $\p^e$ exactly divides $D(K/k)$ 
for some $e \ge 2$ depending on $d$,
while if $p >3$, then $\p$ exactly divides $D(K/k)$.

We  proceed to say more about the value of $e$.
To distinguish the prime ideal $\p$ in the case $p=2$, call it $\p_2$.
The discriminant of
the $k$-basis $\{1, \sqrt{d} \}$ for $K$ is $4d$, so that
$D(K/k)$ divides $4d$.  
First suppose that $d \equiv 3 \pmod{4}$.
Since $\p_2$ divides $(2)$ to the first power in $k$,
$\p_2$ divides $(4d)$ to the second power.  
Thus $e \le 2$ in this case, so that $e=2$.
Next suppose that $d \equiv 2 \pmod{4}$.
Then $8$ divides $4d$, so that $e \in \{2,3\}$ in this case.

So far we have shown that
\begin{equation}\label{2.7}
D(K/k)= \p_2^e \prod_{3< p \in S} \p,
\end{equation}
where $\p_2$ is to be interpreted as $1$ when $d \equiv 1 \pmod{4}$.
(No prime ideal above $3$ occurs in this product
since $3$ does not divide $4d$.)
Taking absolute norms on both sides of (2.7), we have
\begin{equation}\label{2.8}
D(K)/D(k)^2 =
\begin{cases}
d, \quad \quad & d \equiv 1 \pmod{4} \\
4d,   \quad  \quad  & d \equiv 3 \pmod{4} \\
2^{e-1} d, \quad  \quad  & d \equiv 2 \pmod{4}.
\end{cases}
\end{equation}
By Lemma 2.5,
\begin{equation}\label{2.9}
D(K) =
\begin{cases}
9 d^3, \quad \quad & d \equiv 1 \pmod{4} \\
9\cdot 2^6 d^3, \quad  & d \equiv 2,3 \pmod{4}.
\end{cases}
\end{equation}
Thus
\begin{equation}\label{2.10}
D(k)^2=
\begin{cases}
9 d^2, \quad \quad & d \equiv 1 \pmod{4} \\
9 \cdot 2^4 d^2, \quad  \quad  & d \equiv 3 \pmod{4} \\
9 \cdot 2^{7-e} d^2, \quad  \quad  & d \equiv 2 \pmod{4}.
\end{cases}
\end{equation}
This shows that $e$ must be odd, so $e=3$.  Finally,
since $D(k)$ is negative \cite[Prop. 2.15]{N},
we obtain the desired result
\begin{equation}\label{2.11}
D(k) =
\begin{cases}
-3d, \quad \quad & d \equiv 1 \pmod{4} \\
-12d,  \quad  & d \equiv 2,3 \pmod{4}.
\end{cases}
\end{equation}
\end{proof}

\section{Proof of Theorem 1.2}
Assume throughout this section that $v \in M$.
Write $d=3m$ so that $F= \Q(\sqrt{-m})$.
Note that the order $\Z[\sqrt{-3d}]=\Z[3\sqrt{-m}]$ has conductor $3$ or $6$.
The proof of Theorem 1.2 utilizes the following two lemmas.
\begin{lem}\label{Lemma 3.1}
When $3 \mid d$, the extension $F(t)/F$ is unramified.
\end{lem}
\begin{proof}
Let $J$ denote the ring class field for the order $\Z[\sqrt{-m}]$ in $F$.
Note that this order has conductor $1$ or $2$.  The formula for class
numbers of orders \cite[Thm. 7.24]{Cox} shows that the extension
$M/J$ has degree $2$ or $4$.  Note that $t \in M$, since
$v \in M$.  

Since $t$ has degree $3$ over $\Q$ and $J(t) \subset M$, 
we have $|J(t)/J| \le 2$.
Assume for the purpose of contradiction that equality holds.
The cubic minimal polynomial of $t$ over $\Q$ is divisible over $J$ by
the quadratic minimal polynomial of $t$ over $J$.
Therefore some conjugate of $t$ lies in $J$, so that $J(t)=J$.
Thus the assumption is false, and $t \in J$.

Consider the tower $F \subset F(w) \subset F(v)=L$.
The extensions $F(w)/F$ and $L/F(w) = \Q(w)(v)/\Q(w)(\sqrt{d})$ 
cannot ramify at
any rational prime other than $3$.  Thus the same is true
of the extension $F(t)/F$.
Since $F(t) \subset J$ and $3$ does not divide the conductor of 
$\Z[\sqrt{-m}]$,
the extension $F(t)/F$ must be unramified.
\end{proof}

\begin{rem}[\bf{3B}]
By Lemma 3.1, $F(t)/F$ is a cyclic unramified cubic extension
when $3 \mid d$.  Thus
$F(t)$ lies in the Hilbert class field of $F$, so that $3$ divides
the class number of $F$. 
\end{rem}
\begin{rem}[\bf{3C}]
For $d=3m$, consider the principal prime ideal
$\p = (x+y \sqrt{-m})$ in $F$ of norm $p$.  By Lemma 3.1,
the  corresponding Artin symbol for the extension $L/F$ fixes $t$,
so it maps $v$ to either $v$ or $v'$.  In the first case, 
$p \equiv 1 \pmod{3}$ and $v^{p-1} \equiv 1 \pmod{\PP}$,
while in the second case, since here $vv'=1$, we have
$p \equiv -1 \pmod{3}$ and $v^{p+1} \equiv 1 \pmod{\PP}$,
where $\PP$ is a prime ideal of $L$ above $\p$.
\end{rem}

\begin{rem}[\bf{3D}]
When $3 \mid d$ and $v \in M$, $D(L/F)$ has norm $729$. 
To see this, first observe that by Lemma 3.1, $F(t)/Q$ is unramified
at $3$. Note that  $L=F(t)(\sqrt{-3})$. It follows that
$D(L/F(t)) = (3)$, since by \cite[p. 685]{V}, 
$D(L/F(t))$ is the product of the prime ideals in the factorization
of $(3)$ in $F(t)$ which divide $(3)$ to the first power.
Taking the norm, we obtain  $D(L)/D(F(t))^2 = 3^6$.
By Lemma 3.1, $D(F(t))=D(F)^3$, so that $D(L)/D(F)^6 = 3^6$.
This proves that $D(L/F)$ has norm $729$.
\end{rem}

\begin{lem}\label{Lemma 3.2}
Suppose that $3 \mid d$.
The only rational primes that ramify in $L$ are the ones that divide
$D(C)$.  For a rational prime $p$ dividing $D(C)$,
each prime ideal in the factorizations of $(p)$
in $K$ and $L$ occurs to the second power.
Consequently the extension $K/C$ is unramified.
\end{lem}
\begin{proof}
If $p$ divides $D(C)$, then $p$ ramifies in $C$ with ramification
index $2$. First assume that $p \ne 3$.
Then $L/C$ is unramified at $p$, since $K/C$ and $L/K$ are
unramified at $p$.
Thus each prime ideal in the factorizations of $(p)$
in $K$ and $L$ occurs to the second power.
Next consider the case $p=3$.  
We know that $3$ ramifies in $K$ and in $L$, because it ramifies in $C$.
By Lemma 3.1,  $3$ does not ramify
in $F(t)$.  Since $F(t)$ is a subfield of $L$ of index $2$,
it follows that 
each prime ideal in the factorizations of $(3)$
in $K$ and $L$ occurs to the second power.
Thus when $p$ divides $D(C)$, there can be no further ramification
from $C$ to $L$, so that $L/C$ and $K/C$ are unramified.
\end{proof}

We are now prepared for the proof of Theorem 1.2.
\begin{proof}
Consider the set  $S$ of rational primes $p$ that divide
the discriminant $D(C)$. Note that $3 \in S$. By Lemma 3.2,
each prime ideal in the factorization of $(p)$ in $K$
occurs to the second power.
Each prime ideal in the factorization of $(p)$ in $k$
must occur to either
the first or second power, and those occurring to the first power are
exactly the ones that ramify in $K$.

Mimicking the proof of (2.7), we obtain
\begin{equation}\label{3.1}
D(K/k)= 3\p_2^e \prod_{3< p \in S} \p.
\end{equation}
The reason for the factor $3$ is as follows.  
The proof of Theorem 1.1 applies here to show that all primes in $S$ 
other than $3$ ramify in $k$.
However, $3$ does not ramify in $k$; this is due to Lemma 3.1,
since $k \subset F(t)$.
Consequently each prime ideal in the factorization of $(3)$ in $k$
occurs to the first power, so each ramifies tamely in $K$.
This explains the factor $3$ in (3.1).

Taking absolute norms on both sides of (3.1), we have
\begin{equation}\label{3.2}
D(K)/D(k)^2 =
\begin{cases}
9d, \quad \quad & d \equiv 1 \pmod{4} \\
4\cdot 9d,   \quad  \quad  & d \equiv 3 \pmod{4} \\
2^{e-1} 9d, \quad  \quad  & d \equiv 2 \pmod{4}.
\end{cases}
\end{equation}
By Lemma 3.2, $D(K)=D(C)^3$ so that
\begin{equation}\label{3.3}
D(K) =
\begin{cases}
 d^3, \quad \quad & d \equiv 1 \pmod{4} \\
2^6 d^3, \quad  & d \equiv 2,3 \pmod{4}.
\end{cases}
\end{equation}
Thus
\begin{equation}\label{3.4}
D(k)^2=
\begin{cases}
m^2, \quad \quad & d \equiv 1 \pmod{4} \\
2^4 m^2, \quad  \quad  & d \equiv 3 \pmod{4} \\
2^{7-e} m^2, \quad  \quad  & d \equiv 2 \pmod{4}.
\end{cases}
\end{equation}
This shows that $e$ must be odd, so $e=3$.  Finally,
since $D(k)$ is negative,
we obtain the desired result
\begin{equation}\label{3.5}
D(k) =
\begin{cases}
-m, \quad \quad & d \equiv 1 \pmod{4} \\
-4m,  \quad  & d \equiv 2,3 \pmod{4}.
\end{cases}
\end{equation}
\end{proof}

\section{Proof of Theorem 1.3}
Assume throughout this section that $v \notin M$ and $3 \nmid d$.
The proof of Theorem 1.3 utilizes the following five lemmas.
\begin{lem}\label{Lemma 4.1}
Suppose that $3 \nmid d$.
If $D(C)$ is divisible by the prime $q$, then
each prime ideal in the factorizations of $(q)$
in $F(t)$, $K$, and $L$ occurs to the second power.
On the other hand, if $D(C)$ is not divisible by $q$
and $q \ne 3$, then $q$ is unramified in $L$.
\end{lem}
\begin{proof}
First suppose that $q \nmid D(C)$ and $q \ne 3$.
The extension $C/\Q$ is unramified at $q$.
The extension $K/C$ is also unramified at $q$
since it can ramify only at $3$, due to the fact that
$x^3 - v^3$ has discriminant $-27v^6$.
Moreover, $L/K$ can ramify only at $3$, since $L=K(w)$
and $x^2+x+1$ has discriminant $-3$.  Thus $q$
is unramified in $L$.

Next suppose that $q$ divides $D(C)$.  Then $q$
ramifies in $C$ and in $F$, but there can be no
further ramification in the extension $L/C$, since $q \ne 3$. Thus
each prime ideal in the factorizations of $(q)$
in $F(t)$, $K$, and $L$ occurs to the second power.
\end{proof}

\begin{lem}\label{Lemma 4.2}
When $3 \nmid d$,
the extensions $L/F(t)$ and $B/F$ are unramified.
\end{lem}
\begin{proof}
Any prime ideal in $F(t)$ ramifying in
$L=F(t,w)=F(t,\sqrt{d})$ would have to divide
both $F(t)$-basis discriminants $4d$ and $-3$,
which is impossible.
The extension $B/F$ is unramified for the same reason.
\end{proof}

\begin{lem}\label{Lemma 4.3}
When $3 \nmid d$,
the extension $L/B$ ramifies only at $3$.
\end{lem}
\begin{proof}
In view of Lemmas 4.1 and 4.2,  we need only show that
the extension $L/B$ is ramified.
Suppose for the purpose of contradiction that it isn't.
Then by Lemma 4.2, the extension $L/F$ would be unramified.
Consequently  $L \subset M$, which contradicts $v \notin M$.
\end{proof}

\begin{lem}\label{Lemma 4.4}
When $3 \nmid d$, $D(L/B)$ has norm $3^8$ and $D(K/C)$ has norm $3^6$.
\end{lem}
\begin{proof}
We have the prime ideal factorization $(3) = (\sqrt{-3})^2$ in $\Q(\sqrt{-3})$.
In $B$, either $(\sqrt{-3})=\q$ or $(\sqrt{-3})=\p\p'$, where
$\q$ has norm $9$ and $\p, \p'$ have norm $3$.
By Lemma 4.3, these prime ideals ramify wildly 
in the cubic cyclic extension $L/B$,
with ramification index $3$. Thus for some integer $s \ge 3$,
$D(L/B)=\q^s$ or $D(L/B)=(\p\p')^s$ \cite[Cor.2, p. 260]{N}.
Consequently $D(L/B)$ has norm $9^s$. Equivalently,
\begin{equation}\label{4.1}
D(L)=9^s D(B)^3.
\end{equation}

We now know that the ramification index of $3$ in 
the Galois extension $L/\Q$ is 
divisible by $3$. 
In $C$, either $(3)=\q$ or $(3)=\p\p'$, where
$\q$ has norm $9$ and $\p, \p'$ have norm $3$.
As $K/C$ is a cubic extension,  we have wild ramification in $K$ of the form
$(3)=\QQ^3$ or $(3)=(\PP\PP')^3$, where
$\QQ$ has norm $9$ and $\PP, \PP'$ have norm $3$.
It follows that $D(K/C)$ equals $\q^r$ or 
$(\p \p')^r$ for some integer $r \ge 3$.
Thus $D(K/C)$ has norm $9^r$
Moreover, since $D(K/C)$ divides $27v^6$, the norm of $D(K/C)$
divides $27^2=9^3$.Thus $r=3$, so that $D(K/C)$ has the desired norm $3^6$.
Equivalently, 
\begin{equation}\label{4.2}
D(K)=3^6 D(C)^3 =  - 3^3 D(F)^3.
\end{equation}

By Lemma 4.2,  $D(B)=D(F)^2$.  Thus, by (4.1),
\begin{equation}\label{4.3}
D(L)=9^s D(F)^6.
\end{equation}

Since $L=K(\sqrt{-3})$,
$D(L/K)$ equals the product of the prime ideals
in the factorization of $(3)$ in $K$ that occur to odd powers  
\cite[p. 685]{V}.
Thus, in the notation above, $D(L/K)$ equals $\QQ$ or $\PP\PP'$. 
Consequently, $D(L/K)$ has norm $9$, so that by (4.2),
\begin{equation}\label{4.4}
D(L) = 9 D(K)^2= 3^8 D(F)^6.
\end{equation}
Comparing equations (4.3) and (4.4), we see that $s=4$.
Therefore, by (4.1), $D(L/B)$ has the desired norm $3^8$.
\end{proof}

\begin{rem}[\bf{4E}]
Since each prime ideal factor of $D(L/B)$
occurs to the power $s=4$, the corresponding
higher ramification groups $G_i$ must be trivial
for $i \ge 2$, in view of
\cite[p. 265]{N}.   This is because  $s=4=(|G_0|-1) + (|G_1|-1)$,
where $G_0=G_1=\Gal(L/B)$.
\end{rem}

\begin{lem}\label{Lemma 4.5}
When $3 \nmid d$,
$D(F(t)) = 3^4 D(F)^3$.
\end{lem}

\begin{proof}
By (4.4),
$D(L)=3^8 D(F)^6$.   By Lemma 4.2,
$D(F(t))^2=D(L)$, so that
$D(F(t)) = 3^4 D(F)^3$.
\end{proof}

\begin{rem}[\bf{4F}]
When $3 \nmid d$ and $v \notin M$, 
it follows from (4.4) that $D(L/F)$ has norm $3^8$.
This is in contrast with
the case $3 \mid d$; see Remark (5G).
\end{rem}

We are now prepared for the proof of Theorem 1.3.
\begin{proof}
Since $F(t)=k(\sqrt{-3d})$, the odd part of $D(F(t)/k)$
is the product of the prime ideals of odd norm in the factorization
of $(-3d)$ in $k$ which occur to odd powers \cite{V}.
Let $p$ be any odd prime dividing $3d$. By \cite[Prop. 2.13]{N}, $p | D(k)$,
since the minimal polynomial of $t$ over $\Q$ has
discriminant $-108db^2$, which $p$ divides to an odd power. 
Thus
$p$ ramifies in $k$, so that exactly one prime ideal
$\p_p$ in the factorization of $(p)$ in $k$
occurs to an odd power, where $\p_p$ has norm $p$.
Therefore the odd part of $D(F(t)/k)$ equals $\prod \p_p$,
where $p$ runs through the odd prime factors of $3d$.

We now consider the case where $2 \mid D(C)$,
i.e., $d \e 2,3 \pmod 4$.
By Lemma 4.1, each prime ideal in the factorizations of $(2)$
in $F(t)$ and $K$ occurs to the second power.
Each prime ideal in the factorization of $(2)$
in $k$ must occur to the first or second power,
and those occurring to the first power are exactly the ones
that ramify in $F(t)$ and in $K$.  If every prime ideal factor
of $(2)$ in $k$ were to occur to the first power, then $(2)$
would divide $D(K/k)$, contradicting Lemma 2.3.  Thus exactly
one prime ideal $\p_2$ in $k$ divides $(2)$ to the first
power, and $\p_2$ has norm $2$.
Then $\p_2^e$ exactly divides $D(F(t)/k)$ for some $e \ge 2$ depending on $d$.

The discriminant of
the $k$-basis $\{1, \sqrt{-3d} \}$ for $F(t)$ is $-12d$, so that
$D(F(t)/k)$ divides $12d$.
First suppose that $d \equiv 3 \pmod{4}$.
Since $\p_2$ divides $(2)$ to the first power in $k$,
$\p_2$ divides $(12d)$ to the second power.
Thus $e \le 2$ in this case, so that $e=2$.
Next suppose that $d \equiv 2 \pmod{4}$.
Then $8$ divides $12d$, so that $e \in \{2,3\}$ in this case.

So far we have shown that
\begin{equation}\label{4.5}
D(F(t)/k)= \p_2^e \prod_{p} \p_p,
\end{equation}
where $p$ runs through the odd primes dividing $3d$, and
where $\p_2$ is to be interpreted as $1$ when $d \equiv 1 \pmod{4}$.
Taking norms on both sides of (4.5), we have
\begin{equation}\label{4.6}
D(F(t))/D(k)^2 =
\begin{cases}
-3d, \quad \quad & d \equiv 1 \pmod{4} \\
-12d,   \quad  \quad  & d \equiv 3 \pmod{4} \\
-3d \cdot 2^{e-1} , \quad  \quad  & d \equiv 2 \pmod{4}.
\end{cases}
\end{equation}
By (4.6) and Lemma 4.5,
\begin{equation}\label{4.7}
D(k)^2=
\begin{cases}
3^6 d^2, \quad \quad & d \equiv 1 \pmod{4} \\
3^6 \cdot 2^4 d^2, \quad  \quad  & d \equiv 3 \pmod{4} \\
3^6 \cdot 2^{7-e} d^2, \quad  \quad  & d \equiv 2 \pmod{4}.
\end{cases}
\end{equation}
This shows that $e$ must be odd, so $e=3$.  Finally,
since $D(k)$ is negative \cite[Prop. 2.15]{N},
we obtain the desired result
$D(k) = 9 D(F)$.
\end{proof}

\section{Proof of Theorem 1.4}

Assume throughout this section that $v \notin M$ and $d = 3m$.
Then  $F= \Q(\sqrt{-m})$.
Note that the order $\Z[\sqrt{-3d}]=\Z[3\sqrt{-m}]$ 
in $F$ has conductor $3$ or $6$.
The proof of Theorem 1.4 utilizes the following three lemmas.

\begin{lem}\label{Lemma 5.1}
Let $p$ be a prime dividing $D(F)$.  Then each prime ideal
in the factorizations of $(p)$ in $F(t)$, $K$, and $L$
occurs to the second power.
\end{lem}
\begin{proof}
The prime $p$ ramifies in $C$ and in $F$ with ramification index $2$.
Since $L=K(w)$ and $x^2 + x + 1$ has discriminant $-3$ and $p \ne 3$,
the extension $L/K$ is unramified at $p$.   Similarly, $K/C$ is
unramified at $p$, since $x^3-v^3$ has discriminant $-27v^6$.
Thus each prime ideal
in the factorization of $(p)$ in $L$ 
occurs to the second power, and the result follows.
\end{proof}

\begin{lem}\label{Lemma 5.2}
The extension $F(t)/F$ is ramified at $3$.
\end{lem}
\begin{proof}
Suppose for the purpose of contradiction that the lemma is false.
Then by Lemma 5.1, the cubic cyclic extension $F(t)/F$ is unramified.
Consequently, $F(t)$ is contained in the Hilbert class field
of $F$.  Thus $t$ is contained in the ring class field $J$
for the order $\Z[\sqrt{-m}]$,
so that the Artin symbol in Remark (3C) fixes $t$. 
As noted in Remark (3C),
the Artin map fixes $v$ when  $p = x^2 +m y^2 \equiv 1 \pmod{3}$,
which occurs for example when $3|y$.
It follows that the primes $p = X^2 +9 m Y^2$ split
completely in $L$.  These are the primes that split completely in $M$.
By \cite[Thm. 8.19]{Cox}, we have $L \subset M$,
so that $v \in M$, as contradiction.
\end{proof}

\begin{lem}\label{Lemma 5.3}
The norm of  $D(F(t)/F)$ is equal to $3^8$.
\end{lem}
\begin{proof}

In $F$, either $(3)=\q$ or $(3)=\p\p'$, where
$\q$ has norm $9$ and $\p, \p'$ have norm $3$.
By Lemma 5.2,
these prime ideals ramify wildly in $F(t)$
with ramification index $3$.  Thus
\begin{equation}\label{5.1}
D(F(t)/F) = \q^s \ \mbox{or}\  D(F(t)/F) = (\p\p')^s
\end{equation}
for some integer $s$ with $3 \le s \le 5$ \cite[pp. 260,262]{N}.
In either case, $D(F(t)/F)$ has norm  $9^s$.
Since $F(t)/F$ is a cubic cyclic extension, 
$D(F(t)/F)$ is equal to the square of an ideal in $F$ \cite[Cor. 2, p. 266]{N}.
Thus by (5.1), $s=4$, so $D(F(t)/F)$ has norm  $3^8$.
\end{proof}

We are now prepared for the proof of Theorem 1.4.
\begin{proof}
Let $p$ be a prime dividing $D(F)$.   By Lemma 5.1,
each prime ideal in the factorization of $(p)$ in $k$
must occur to the first or second power, and those
that occur to the first power are exactly the ones that
ramify in $F(t)$ and in $K$.  If every prime ideal factor
of $(p)$ in $k$ were to occur to the first power, then
$(p)$ would divide $D(K/k)$, contradicting Lemma 2.3.
Thus $p$ ramifies in $k$, so that exactly one prime ideal
$\p_p$ in the factorization of $(p)$ in $k$ occurs to an
odd power, and $\p_p$ has norm $p$.  Since 
$F(t) = k(\sqrt{-m})$, the odd part of 
$D(F(t)/k)$ is the product of the prime ideals of odd norm
in the factorization of $(m)$ in $k$ which occur to odd powers
\cite{V}. Thus the odd part of $D(F(t)/k)$ equals
$\prod \p_p$, where $p$ runs through the odd prime factors of $m$.

Consider now the case when $p=2$.
Since $\p_2$ ramifies wildly in $F(t)$,
$\p_2^e$ exactly divides $D(F(t)/k)$ 
for some integer $e \ge 2$.
Arguing as in the paragraph above (2.7), we see that
$e=2$ when $d \equiv  3 \pmod{4}$ and $e \in \{2,3\}$
when $d \equiv 2 \pmod{4}$. Thus
\begin{equation}\label{5.2}
D(F(t)/k) = \p_2^e \prod_{p} \p_p ,
\end{equation}
where $p$ runs through the odd prime factors of $m$, and
where $\p_2$ is to be interpreted as $1$ when $d \equiv 1 \pmod{4}$.

Taking norms on both sides of (5.2), we have
\begin{equation}\label{5.3}
D(F(t))/D(k)^2 =
\begin{cases}
-m, \quad \quad & d \equiv 1 \pmod{4} \\
-4m,   \quad  \quad  & d \equiv 3 \pmod{4} \\
-2^{e-1} m, \quad  \quad  & d \equiv 2 \pmod{4}.
\end{cases}
\end{equation}
By Lemma 5.3,
\begin{equation}\label{5.4}
D(F(t)) = 3^8 D(F)^3 =
\begin{cases}
-m^3 3^8, \quad \quad & d \equiv 1 \pmod{4} \\
-m^3 2^6 3^8, \quad  & d \equiv 2,3 \pmod{4}.
\end{cases}
\end{equation}
By (5.3) and (5.4),
\begin{equation}\label{5.5}
D(k)^2=
\begin{cases}
m^2 3^8, \quad \quad & d \equiv 1 \pmod{4} \\
2^4 m^2 3^8, \quad  \quad  & d \equiv 3 \pmod{4} \\
2^{7-e} m^2 3^8, \quad  \quad  & d \equiv 2 \pmod{4}.
\end{cases}
\end{equation}
This shows that $e$ must be odd, so $e=3$.  Finally,
since $D(k)$ is negative \cite[Prop. 2.15]{N},
we obtain the desired result $D(k) = 81 D(F)$.
\end{proof}

\begin{rem}[\bf{5G}]
When $3 \mid d$ and $v \notin M$, $D(L/F)$ has norm $3^{18}$.
To see this, 
note that by the proof of Lemma 5.3, the 
prime ideal factorization of $(3)$
in $F(t)$ is either $\QQ^3$ or $(\PP\PP')^3$,
where $\QQ$ has norm $9$ and $\PP, \PP'$ have norm $3$.
Since these prime ideals occur to the odd power $3$,
and since $L=F(t)(\sqrt{-3})$, it follows from \cite{V} that
$D(L/F(t))$ equals $\QQ$ or $\PP\PP'$.  In either case,
$D(L/F(t))$ has norm $9$, so that $D(L) = 9 D(F(t))^2$.
By Lemma 5.3, $D(F(t))=3^8 D(F)^3$.
Combining these last two equalities, we obtain
the desired result $D(L) = 3^{18} D(F)^6$.
\end{rem}

\section{Ring class and ray class fields}
Recall that $v$ is the real cube root of the 
fundamental unit $u=a + b \sqrt{d}$.
The following theorem gives 
explicit criteria in terms of $a$ and $b$ for $v$ to lie 
in the ring class field $M$.
\begin{thm}\label{Thm 6.1}
We have $v \in M$ if and only if
\begin{equation}\label{6.1}
a \equiv 0 \pmod{9} \quad \mbox{or} \quad 
\begin{cases}
a \equiv \pm 2 \pmod{9}  \quad & \ \mbox{when} \ \ a^2 - d b^2 =-1 \\
a \equiv \pm 1 \pmod{27}  \quad & \ \mbox{when} \ \ a^2 - d b^2 =+1.
\end{cases}
\end{equation}
\end{thm}
\begin{proof}
By the proof of \cite[Thm. 6]{Herz},  (6.1) holds
if and only if $F(t)/F$ is unramified.
The result now follows from Lemmas 2.1, 3.1, 4.5, and 5.3.
\end{proof}

Recall that $M_3$ denotes the ring class field of $F$
for the order $\Z[\sqrt{-27d}]$.  It follows from
\cite[Thm. 7.24]{Cox}  that $|M_3:M|=3$.
Consider the example $v = (5/2 + \sqrt{21}/2)^{1/3}$ for $d=21$.
By Theorem 6.1, $v$ does not lie in $M$, since
$a = 5/2 \equiv 16 \not\equiv \pm 1 \pmod{27}$.
On the other hand, this $v$ does lie in $M_3$.
In fact, Theorem 6.2 shows that $v \in M_3$ for every $v$,
i.e., every $u_p$ is a cubic residue mod the primes $p=x^2 +27dy^2$.
This fact had been conjectured by the first author, and the proof is due
to the third author.
For an extension of Theorem 6.2, see Conjecture 7.8.

Note that for each squarefree $d>1$, 
the fundamental unit $u$ can be written in the form
$u = (m + n\sqrt{d})/2$ with 
nonzero integers $m, n$ such that  $m^2 - dn^2 =\pm4$.
\begin{thm}\label{Thm 6.2}
Every $v$ lies in $M_3$. 
\end{thm}
\begin{proof}
We will utilize the  
integral quadratic forms below when $d \equiv 1 \pmod{4}$:
\begin{equation}\label{6.2}
p= A^2 +3dB^2 \Longrightarrow p=x^2 + xy + \frac{3d+1}{4} y^2
\end{equation}
and
\begin{equation}\label{6.3}
p= A^2 +27dB^2 \Longrightarrow p=x^2 + xy + \frac{27d+1}{4} y^2,
\end{equation}
where $x=A-B$ and $y=2B$.

\noindent {\bf Case 1:}
$m^2 - dn^2 = -4$.   

\noindent
In this case we cannot have $d \equiv 3 \pmod{4}$.
First assume that $9\mid m$ and $p= A^2 +3dB^2$.
If $d \equiv 2 \pmod{4}$, then $v \in M$
by \cite[Thm. 5.1]{S} with $k=2$.
If $d \equiv 1 \pmod{4}$, then $v \in M$
by (6.2) and \cite[Thm. 5.1]{S} with $k=1$.

Next assume that $9\nmid m$ and $p= A^2 +27dB^2$.
If $d \equiv 2 \pmod{4}$, then $v \in M_3$
by \cite[Thm. 5.1]{S} with $k=6$.
If $d \equiv 1 \pmod{4}$, then $v \in M_3$
by (6.3) and \cite[Thm. 5.1]{S} with $k=3$.
This completes the proof in Case 1.

\noindent {\bf Case 2:}
$m^2 - dn^2 = 4$.   

\noindent
First assume that $9\mid m$ and $p= A^2 +3dB^2$.
If $d \equiv 2,3 \pmod{4}$, then $v \in M$
by \cite[Thm. 5.3]{S} with $k=2$.
If $d \equiv 1 \pmod{4}$, then $v \in M$
by (6.2) and \cite[Thm. 5.3]{S} with $k=1$.

Next assume that $9\nmid m$.
Since $(m-2)(m+2)=dn^2$, we may choose the sign of $m$ so that
$\ord_3(m-2) \ge \ord_3 n$, where $\ord_3$ denotes the $3$-adic order.
There is no loss of generality in fixing this sign, since
the conjugate $v'$ of $v$ satisfies $vv'=\pm1$. 
We will consider separately the cases $d\equiv 2,3\pmod 4$,
$d\equiv 5\pmod 8$, and $d\equiv 1\pmod 8$.
\par  If $d\equiv 2,3\pmod 4$, $9\nmid \frac{m-2}{(m-2,n)}$ and $p=A^2+27dB^2$,
then $v \in M_3$ by \cite[Thm. 5.5]{S} with $k=2 \cdot 3$.
If $d\equiv 2,3\pmod 4$, $9\mid \frac{m-2}{(m-2,n)}$ and $p=A^2+3dB^2$, 
then $v \in M$ by \cite[Thm. 5.5]{S} with $k=2 \cdot 1$.
(We can ignore the restriction $p \nmid n$ in 
\cite[Thm. 5.5]{S} because $u \equiv 1 \pmod p$ when $p \mid n$.)
\par  If $d\equiv 5\pmod 8$, $9\nmid \frac{m-2}{(m-2,n)}$ and $p=A^2+27dB^2$,
then $v \in M_3$
by (6.3) and \cite[Thm. 5.5]{S} with $k=1 \cdot 3$.
If $d\equiv 5\pmod 8$, $9\mid \frac{m-2}{(m-2,n)}$ and $p=A^2+3dB^2$,
then $v \in M$
by (6.2) and \cite[Thm. 5.5]{S} with $k=1 \cdot 1$.
\par From now on assume that $d\equiv 1\pmod 8$. 
Set $r=\ord_2\frac{4(m-2)}{(m-2,n)^2}$.
If $r>0$, $r\equiv 0,1\pmod 3$,  
$9\nmid \frac{m-2}{(m-2,n)}$ and $p=A^2+27dB^2$,
then $v \in M_3$
by \cite[Thm. 5.5]{S} with $k=2 \cdot 3$.
If $r>0$, $r\equiv 0,1\pmod 3$,  
$9\mid \frac{m-2}{(m-2,n)}$ and $p=A^2+3dB^2$,
then $v \in M$
by \cite[Thm. 5.5]{S} with $k=2 \cdot 1$.

Finally consider the case where either $r=0$ or $r\equiv 2\pmod 3$.
If $9\nmid \frac{m-2}{(m-2,n)}$ and $p=A^2+27dB^2$,
then $v \in M_3$
by (6.3) and \cite[Thm. 5.5]{S} with $k=1 \cdot 3$.
If $9\mid \frac{m-2}{(m-2,n)}$ and $p=A^2+3dB^2$,
then $v \in M$
by (6.2) and \cite[Thm. 5.5]{S} with $k=1 \cdot 1$.
This completes the proof in Case 2.
\end{proof}

Let $\OO$ denote the order $\Z[\sqrt{-3d}]$ in $F$.
Let $F_{(N)}$ denote the ray class field of $F$ modulo $(N)$.
Define the class field $F_{(N),\OO}$ 
as in \cite[p. 853]{Cho}.
By \cite[Thm. 1]{Cho} with $n=3d$,
\begin{equation}\label{6.4}
M = F_{(1),\OO}, \quad M_3 = F_{(3),\OO}.
\end{equation}
By \cite[Thm. 1]{Cho},
\begin{equation}\label{6.5}
F_{(N)} \subset  F_{(N),\OO} \subset  F_{(NT)},
\end{equation}
where $T$ denotes the conductor of the order $\OO$.
Theorem 6.3 relates the ring class fields $M$ and $M_3$
to ray class fields of $F$.
\begin{rem}[\bf{6H}]
Analogous to (6.4), we have $M_6 = F_{(6),\OO}$,
where $M_6$ is the ring class field of $F$ for the
order $\Z[\sqrt{-108d}]$.
\end{rem}
\begin{thm}\label{Thm 6.3}
If $3 \nmid d$, then 
\begin{equation}\label{6.6}
\begin{cases}
M_3=F_{(6)}  \quad & \ \mbox{when} \ \ d \equiv 1 \pmod{4} \\
M_3=F_{(3)}  \quad & \ \mbox{when} \ \ d \equiv 2,3 \pmod{4}.
\end{cases}
\end{equation}
If $3\mid d$, then
\begin{equation}\label{6.7}
\begin{cases}
M=F_{(6)}  \quad & \ \mbox{when} \ \ d \equiv 1 \pmod{4} \\
M=F_{(3)}  \quad & \ \mbox{when} \ \ d \equiv 2,3 \pmod{4}.
\end{cases}
\end{equation}
\end{thm}
\begin{proof}
Suppose first that 
$3 \nmid d$ and
$d \equiv 2,3 \pmod{4}$, so that $T=1$.
Then by (6.4) and (6.5) with $N=3$,  we have
$M_3=F_{(3)}$.  Next suppose that
$3 \nmid d$ and
$d \equiv 1 \pmod{4}$, so that $T=2$.
Then by (6.4) and (6.5) with $N=3$,  we have
$M_3 \subset F_{(6)}$.
To show that $M_3=F_{(6)}$,
it suffices to show that
\begin{equation}\label{6.8}
|M_3:F| = |F_{(6)}:F|.
\end{equation}
Using the formula $\phi((6))/2 \ $ \cite[p. 50]{Chil}
for the ray class number on the right,
and the formula \cite[Thm. 7.24]{Cox}
for the ring class number on the left,
we see that both members of (6.8) are equal to
\begin{equation}\label{6.9}
\begin{cases}
3h_F,  \quad & \ \mbox{when} \ \ d \equiv 5 \pmod{8} \\
9h_F, \quad & \ \mbox{when} \ \ d \equiv 1 \pmod{8},
\end{cases}
\end{equation}
where $h_F$ is the class number of $F$.
This completes the proof of (6.6).

Now suppose that $d = 3m$, so that $F=\Q(\sqrt{-m})$
and $\OO = \Z[3 \sqrt{-m}]$.
Let $g$ denote the ordered pair of Kronecker symbols
$g=((-m/3),(-m/2))$.

First assume that $d \equiv 1 \pmod{4}$, so that $T=6$.
Then by (6.4) and (6.5) with $N=1$,  we have
$M \subset F_{(6)}$.
To show that $M=F_{(6)}$,
it suffices to show that
\begin{equation}\label{6.10}
|M:F| = |F_{(6)}:F|.
\end{equation}
Appealing again to \cite{Cox} and \cite{Chil},
we see that both members of (6.10) are equal to
\begin{equation}\label{6.11}
\begin{cases}
4h_F,  \quad & \ \mbox{when} \ \ g=(-1,1) \\
12h_F, \quad & \ \mbox{when} \ \ g=(-1,-1) \\
2h_F, \quad & \ \mbox{when} \ \ g=(1,1) \\
6h_F, \quad & \ \mbox{when} \ \ g=(1,-1).
\end{cases}
\end{equation}
Finally assume that $d \equiv 2,3 \pmod{4}$, so that $T=3$.
Then by (6.4) and (6.5) with $N=1$,  we have
$M \subset F_{(3)}$.
To show that $M=F_{(3)}$,
it suffices to show that
\begin{equation}\label{6.12}
|M:F| = |F_{(3)}:F|.
\end{equation}
Again by  \cite{Cox} and \cite{Chil},
we see that both members of (6.12) are equal to
\begin{equation}\label{6.13} 
\begin{cases}
2h_F, \quad & \ \mbox{when} \ \ (-m/3)=1 \\
4h_F, \quad & \ \mbox{when} \ \ (-m/3)=-1.
\end{cases}
\end{equation}
This completes the proof of (6.7).
\end{proof}

\section{Generalizations for integers $u \in \Q(\sqrt{d})$ with cubic norms}
For each squarefree $d >1$, let $\OO(d)$ denote the ring of integers
in $\Q(\sqrt{d})$, and write $f(d)$ for the fundamental unit in $\OO(d)$.  
Let $S_d$ denote the set of $u \in \OO(d)$
for which the norm of $u$ is a cube, $u$ is not the cube of an element in
$\OO(d)$, and $u$ is not divisible in $\OO(d)$ by the cube of a rational prime.
(The notation $u$ will no longer be restricted solely
for fundamental units.)
Denote the norm of $u$ by $N(u)=n^3$, and let $\nu$ denote the real cube root
$u^{1/3}$.

Let $S_d^*$ denote the subset of $u \in S_d$ 
such that $p\mid d$ for each rational
prime $p$ dividing $u$.
For example, $S_d^*$ contains the fundamental unit $f(d)$.
More generally, whenever $\mu \in \OO(d)$ is not divisible by a rational prime,
$S_d^*$ contains 
\begin{equation}\label{7.1}
\mu^3 f(d)^{\pm 1}.
\end{equation}
Examples of elements in $S_d^*$ not of the form (7.1) are
\begin{equation}\label{7.2}
17+2\sqrt{79},\ 13+\sqrt{142},\ 14+\sqrt{223},\ (11+\sqrt{229})/2,
\ 28+3\sqrt{235},\ 77+2\sqrt{254},
\end{equation}
whose norms are $-27,\ 27,-27,-27,-11^3, 17^3$, respectively.

Let $P_d$ denote the set of $\mu = a + b\sqrt{d} \in \OO(d)$
with norm $N(\mu)=n^3$ satisfying the following conditions:
\begin{align*}
n \equiv 2, 5, 8 \pmod{9} &\implies
	a \equiv 0, \pm2, \pm7, \pm9, \pm11  \pmod{27};\\
n \equiv 1 \pmod{9} &\implies
        a \equiv 0,\pm1, \pm 9 \pmod{27};\\
n \equiv 4 \pmod{9} &\implies
        a \equiv 0,\pm8, \pm 9 \pmod{27};\\
n \equiv 7 \pmod{9} &\implies
        a \equiv 0,\pm9, \pm10 \pmod{27};\\
n \equiv 0 \pmod{9} &\implies
        a \equiv \pm4, \pm5, \pm13 \pmod{27};\\
n \equiv \pm3 \pmod{9} &\implies
        a \equiv 0, \pm4, \pm5, \pm9, \pm13 \pmod{27}.
\end{align*}
When $u$ is a fundamental unit (so that $n=\pm1$), 
it follows from Theorem 6.1 that
$v \in M$ if and only if $u \in P_d$.  Theorem 7.3 shows that
this equivalence holds for a more general set of $u$ under the condition that
the class number $h(d)$ is not a multiple of $3$.
We conjecture that the condition on the class number can be dropped.
The proof of Theorem 7.3 depends on the following two lemmas.

\begin{lem} \label{Lem 7.1}
Let $u \in S_d$.  Then the principal ideal
$(u)$ factors in $\OO(d)$  as
\begin{equation}\label{7.3}
(u) = \A^3 q \QQ,
\end{equation}
where $\A$ and $\QQ$ are ideals, and
$q$ is a squarefree integer with $q=N(\QQ)$,
where each rational prime factor of $q$ splits in $\OO(d)$.
\end{lem}
\begin{proof}
Let $\p$ be a prime ideal factor of $(u)$
lying above a rational prime $p$,  
so that for some $e \ge 1$, $\p^e || (u)$.
We will use the term
``$p$-part" of $(u)$ to denote the contribution of the
ideals above $p$ to the ideal factorization of $(u)$.
As  $u \in S_d$, we can write $n^3 = N(u) = u u'$,
where $u'$ is the conjugate of $u$ in $\OO(d)$.

Suppose first that $p$ does not split in $\OO(d)$.
Then $\p^{2e}$ exactly divides $(uu')=(n^3)$,
so that $3 \mid e$.  Thus the $p$-part of 
$(u)$ is a cube, which can be absorbed in (7.3)
by $\A^3$.

Now suppose that $p$ splits and write $(p) = \p\p'$.
Since $\p^e || (u)$, taking norms yields $p^e \mid n^3$.
Moreover, if $\p' \nmid (u)$, then $p^e || n^3$,
so that $3 \mid e$ and again the $p$-part of
$(u)$ is a cube that can be absorbed by $\A^3$.
It remains to consider the case when $\p' \mid (u)$.
In this case, $p^i$ divides $u$ for some $i \in \{1,2\}$
with $e \ge i$.
(We cannot have $i>2$ by definition of $S_d$.)
We may assume that $3 \nmid e$, otherwise we revert back
to the previous situations  where the $p$-part is a cube.
We have $\p^{e-i}||(u/p^i)$.
By taking norms, $p^{e-i}||(n^3/p^{2i})$.
Thus $p^{e+i} ||n^3$,so that $3|(e+i)$.
Therefore $e =2i + 3k$ for some $k \in \Z$.
Since $-i \le e-2i = 3k$, we must have $k \ge 0$.
The $p$-part of $(u)$ is $p^i\p^{e-i} = p^i\p^i \p^{3k}$.
The cube $\p^{3k}$ can be absorbed by $\A^3$.
Taking the product of the 
$p$-parts of $(u)$ over all $p$, we could
obtain (7.3), where $\QQ$ is the product of the $\p^i$
and $q=N(\QQ)$ is the product of the $p^i$, but
we need $i=1$ for every $p$ in order to ensure that $q$ is squarefree.
Fortunately, we can rearrange each product $p^2\p^2$
so that the exponents equal 1, using
$p^2\p^2 = p \p' \p^3$.
\end{proof}

Let $R_d$ be the set of $\beta \in \OO(d)$ having the form
$\beta=nr +ns \sqrt{d}$,  where $r +s \sqrt{d}$  has norm $n$
(so that $N(u) = n^3$).  For example, $\beta=(31 + 155\sqrt{5})/2 \in R_5$
with $N(\beta)= -31^3$. An example of an element of $S_d$ that is not
in $R_d$ is $u=1376 +387\sqrt{79}$, for which
$N(u)= -215^3$.

We provide  a  computer-generated proof via Mathematica
for the following lemma.

\begin{lem}\label{Lem 7.2}
Let $\beta, \gamma \in R_d$, $\mu \in \OO(d)$, with norms $N(\mu)$
and $N(\gamma)$ both nonzero modulo $3$. Then
\begin{equation}\label{7.4}
\beta \in P_d \Longleftrightarrow  \beta \mu^3 \in P_d 
\end{equation}
and
\begin{equation}\label{7.5}
(\beta \in P_d \ \mbox{and} \ \gamma \in P_d)
\Longrightarrow  \beta \gamma  \in P_d.
\end{equation}
\end{lem}
\begin{proof}
Write  $\mu = x+y \sqrt{d}$ and $\beta = nr+ns\sqrt{d}$
with $n=r^2 - d s^2$, so that $N(\beta)= n^3$.
Create in Mathematica 
the master set of $9{,}565{,}938$ quintuples $(r,s,x,y,d)$
modulo $27$ for which $x^2 -d y^2$ is nonzero modulo $3$.
Compute the subset of the master set
for which $(nr+ns\sqrt{d})(x+y\sqrt{d})^3 \in P_d$.
This turns out to be exactly the same subset for which
$\beta=nr+ns\sqrt{d} \in P_d$, thereby proving (7.4).

Next, write $\gamma =mx+my\sqrt{d}$
with $m=x^2 - dy^2$, so that $N(\gamma)= m^3$.
Using the same master set, compute the subset for which
$(nr+ns\sqrt{d})(mx+my\sqrt{d}) \in P_d$.
This turns out to contain the subset for which both
of these factors lie in $P_d$, thus proving (7.5).
The details of the Mathematica proof are given in \cite{EVV}.
\end{proof}

For $u \in S_d$,
recall that $\nu$ denotes the real cube root $u^{1/3}$, $n$ denotes the
cube root of the norm of $u$, and $v$ denotes the real cube root of
the fundamental unit $f(d)$.

\begin{thm}\label{Thm 7.3}
Let $u \in S_d^*$ have nonzero norm modulo $3$.
Assume that $3\nmid h(d)$.
Then $\nu \in M$ if and only if $u \in P_d$.
\end{thm}
\begin{proof}
By definition of $S_d^*$, any rational prime dividing $u$
must ramify in $\OO(d)$.  Thus, by Lemma 7.1, we have
$(u) = \A^3$ for some ideal $\A$
in $\OO(d)$.   
Since $3\nmid h(d)$ and both of  $\A^3$ , $\A^{h(d)}$ are principal,
it follows that $\A$ is principal.
This shows that $(u) = (\mu^3)$ for some $\mu \in \OO(d)$,
so that $u$ has the form (7.1). 
Without loss of generality, we will assume the plus sign in (7.1),
and we write  $u = f(d) (x+y \sqrt{d})^3$, where $(x+y \sqrt{d})$
is an element of $\OO(d)$ whose norm is not divisible by $3$.

Clearly $\nu = v(x+y \sqrt{d})\in M$ if and only if $v \in M$.
As noted above Lemma 7.1, $v \in M$ if and only if
$f(d) \in P_d$. Thus $\nu \in M$ if and only if $f(d) \in P_d$.
It remains to prove that
\begin{equation}\label{7.6}
f(d) \in P_d \Longleftrightarrow  f(d)(x+y\sqrt{d})^3  \in P_d .
\end{equation}
This follows from the special case $\beta=f(d)$ of Lemma 7.2.
\end{proof}

Let $M_e$ denote the ring class field for the order $\Z[e\sqrt{-3d}]$ 
in $F=\Q(\sqrt{-3d})$.  In particular, $M_1 = M$. 
For $u \in \OO(d)$, let $c=c(u)$ be the product
of the distinct rational primes which divide $u$ but not $d$.
Note that for $u \in S_d$, we have  $c(u)=1$ if and only if $u \in S_d^*$.

Theorem 7.5 below extends Theorem 7.3.
Just as for Theorem 7.3, 
we conjecture that the condition on the class number can be dropped.
The proof of Theorem 7.5 is conditional on the following conjecture.

\begin{conj}\label{Conj 7.4}
Let $\beta \in \OO(d)$ have a cubic norm with $3\nmid N(\beta)$.
Then
$\beta^{1/3} \in M_c$ if and only if
$\beta \in P_d$, where $c=c(\beta)$.
In particular, this equivalence holds for every  $\beta \in R_d$ 
with $3\nmid N(\beta)$.
\end{conj}

\begin{rem}[\bf{7I}]
Let $\beta, \gamma \in R_d$ with $3 \nmid N(\beta\gamma)$.
If $\beta^{1/3} \in M_{c(\beta)}$
and $\gamma^{1/3} \in M_{c(\gamma)}$, 
then by definition of $M_{c(u)}$, we have
$(\beta\gamma)^{1/3} \in M_{c(\beta\gamma)}$.
This demonstrates that (7.5) is consistent with Conjecture 7.4.
\end{rem}

\begin{thm}\label{Thm 7.5}
Let $u \in S_d$ have nonzero norm modulo $3$.  
Assume that Conjecture 7.4 holds and that $3 \nmid h(d)$.   Then
$\nu \in M_c$ if and only if 
$u \in P_d$, where $c=c(u)$. 
\end{thm}
\begin{proof}
Let $T \in \{h(d),2h(d)\}$ be chosen such that $3 \mid (T-1)$.
By (7.3),
\begin{equation}\label{7.7}
(u^T) = (\A^T)^3 q^T \QQ^T.
\end{equation}
Since $\QQ^T$ is a principal ideal of norm $q^T$,  we have
$(q^T \QQ^T) = (\gamma)$ with $\gamma \in R_d$.
By (7.7), $q=c$, where $c = c(u) = c(\gamma)$.
Since $\A^T$ is principal, we have
$u^T = \mu_0^3 \beta$ for some $\mu_0 \in \OO(d)$ and $\beta \in R_d$.
Therefore $u^{3k}u = \mu_0^3 \beta$  for some $k$.
Multiplying by the conjugate $u'^{3k}$, we obtain
\begin{equation}\label{7.8}
j^3 u = \mu^3 \beta
\end{equation}
for some $\mu \in \OO(d)$, where 
$j=N(u)^k$ is not divisible by 3.  Consequently
\begin{equation}\label{7.9}
j \nu  = \mu \beta^{1/3}.
\end{equation}
Since $j, \mu \in M \subset M_c$, 
\begin{equation}\label{7.10}
\nu \in M_c \Longleftrightarrow \beta^{1/3} \in M_c
\Longleftrightarrow \beta \in P_d,
\end{equation}
where the first equivalence follows from (7.9)
and the second follows from Conjecture 7.4.
By Lemma 7.2 and (7.8),
\begin{equation}\label{7.11}
\beta \in P_d \Longleftrightarrow \mu^3\beta \in P_d
\Longleftrightarrow j^3 u \in P_d \Longleftrightarrow u \in P_d.
\end{equation}
The result now follows from (7.10) and (7.11).
\end{proof}

Some numerical examples supporting Theorem 7.5 with $3 \mid h(d)$
are given in the last section of \cite{EVV}.

\begin{thm}\label{Thm 7.6}
Let $u \in \OO(d)$ have a cubic norm and write $\nu = u^{1/3}$.
Then $\nu \in M_{c}$ if and only if $\nu \in M_{c'}$,
where $c'=c'(u)$ denotes the odd part of $c=c(u)$.
\end{thm}
\begin{proof}
It suffices to prove that 
$\nu \in M_{c}$ implies $\nu \in M_{c'}$ for even $c$.
Assume for the purpose of contradiction that
$\nu \in M_{c}$ but $\nu \notin M_{c'}$.
By \cite[Thm. 7.24]{Cox}, 
$M_{c}$ has degree $2$ over $M_{c'}$.
Thus the minimal polynomial of $\nu$ over $M_{c'}$
is quadratic.   Since this quadratic polynomial must divide
the cubic polynomial $x^3 - u=x^3 - \nu^3$ over $M_{c'}$, it follows that
this cubic polynomial has a linear factor over $M_{c'}$.
Since the cube roots of unity lie in $M$, we obtain the desired
contradiction $\nu \in M_{c'}$.
\end{proof}

We close this section with two conjectures.  
When $u$ is a fundamental unit, Conjecture 7.7 reduces to Theorem 1.5,
while Conjecture 7.8 reduces to Theorem 6.2.
Extensive numerical evidence for Conjectures 7.4, 7.7, and 7.8 
is given in \cite{EVV}.

\begin{conj}\label{Conj 7.7}
Let $u \in S_d$ and let $c'=c'(u)$ denote the odd part of $c=c(u)$.
When $\nu \in M_c$, the
norm of $D(F(\nu)/F)c^{-4}$
equals $1$ or $3^6$ according as $3\nmid d$ or $3 \mid d$, and
when $\nu \notin M_c$, the norm of $D(F(\nu)/F)c^{-4}$
equals $3^8$ or $3^{18}$
according as $3\nmid d$ or $3 \mid d$,
except that when $d \equiv 3 \pmod{4}$,
each $c^{-4}$ is to be replaced by $c'^{-4}$. 
\end{conj}

\begin{conj}\label{Conj 7.8}
For every $u \in S_d$,  we have $\nu \in M_{3c}$, where $c=c(u)$.
\end{conj}

\section{Generalization of Theorem 6.1}
Let $d>1$ be squarefree.
In this section, we compute the value of
$(\frac{m + n\sqrt{d}}{2})^{\frac{p-1}{3}} \pmod{p}$ 
for the primes $p=x^2 +3dy^2$,
where $m^2 - dn^2 =\pm4$. Of course the value is $1 \pmod{p}$
if and only if  $\frac{m + n\sqrt{d}}{2}$ is a cubic residue for these
$p$.  Theorem 6.2 shows that the value is $1 \pmod{p}$ whenever
$3\mid y$, so throughout this section, it will be assumed that $3 \nmid y$.
The signs of $x$ and $y$ will be chosen such that $3\mid (x-y)$.
Theorems 8.1 and 8.2 address the cases when $m^2-dn^2=-4$ 
and $m^2-dn^2=4$, respectively.  In the latter case, we may assume
that $p \nmid n$, since otherwise $m \e \pm2 \pmod p$ so that
$\frac{m + n\sqrt{d}}{2}$ is a cubic residue for $p$.

Set $\omega=\f{-1+\sqrt{-3}}2$.  We will 
utilize properties (8.1)--(8.4) for the cubic Jacobi symbol
\cite[p. 63]{S}.   
For integers $a, b, c, d$ with $3 \nmid c \ $, $(d,c)=1\ $, and
$a-2 \e b \e 0 \pmod 3$,
\begin{equation}\label{8.1}
\Big(\f{\omega}{a+b\omega}\Big)_3=\omega^{\f{a+b+1}3},
\quad \Big(\f{1-\omega}{a+b\omega}\Big)_3=\omega^{\f{2(a+1)}3},
	\quad \Big(\f {d}{c} \Big)_3=1.
\end{equation}
By the cubic reciprocity law,
\begin{equation}\label{8.2}
\Big(\f {a+b\omega}{c+d\omega}\Big)_3=\Big(\f{c+d\omega} {a+b\omega}\Big)_3,
        \quad \t{when} \quad b\e d \e 0 \pmod{3}, \ \ 3\nmid ac.
\end{equation}
When $a-2\e b\e 0\pmod 3$, it follows from (8.1) that
\begin{equation}\label{8.3}
\Big(\f{1+2\omega}{a+b\omega}\Big)_3=\Big(\f{\omega(1-\omega)}{a+b\omega}\Big)_3=\omega^{\f b3}, \quad  \Big(\f{3}{a+b\omega}\Big)_3=\Big(\f{-\omega^2(1-\omega)^2}{a+b\omega}\Big)_3=\omega^{\f{2b}3}.
\end{equation}
If  $3\nmid a$  and $(a^2,c^2+3d^2)=1$, we have
\begin{equation*}
\Big(\f{c+d(1+2\omega)} {a}\Big)_3\Big(\f{-c+d(1+2\omega)} {a}\Big)_3
=\Big(\f{-c^2-3d^2} {a}\Big)_3=1,
\end{equation*}
and so
\begin{equation}\label{8.4}
\Big(\f{-c+d(1+2\omega)} {a}\Big)_3=
\Big(\f{c+d(1+2\omega)} {a}\Big)_3^{-1}.
\end{equation}

We will also need
\begin{equation}\label{8.5}
\Big(\f {m+n}2+m\omega\Big)\Big(\f {m+n}2+m\omega^2\Big)=\f {3m^2+n^2}4.
\end{equation}

\begin{thm}\label{Thm 8.1}
Suppose that $m^2-dn^2=-4$.  Then modulo $p=x^2 + 3dy^2$, we have
	\begin{equation}\label{8.6}
\Big(\f {m+n\sqrt d}2\Big)^{\f {p-1}3}\e 
\begin{cases} 
1
&\hbox{if $m\e 0,  \pm 4 \pmod 9$,} \\
\f  12\big(-1+ (\f {mn/3}3)\f
{x}{dy}\sqrt d\big)
&\hbox{if $m\e \pm 3\pmod 9$,}\\
\f  12\big(-1- (\f{mn}3)\f {x}{dy}\sqrt d\big)
&\hbox{if $m\e \pm 1\pmod 9$,}\\
\f  12\big(-1+ (\f{mn}3)\f {x}{dy}\sqrt d\big)
&\hbox{if $m\e \pm 2\pmod 9$,}
\end{cases}
\end{equation}
	where $\big(\f{\cdot}3\big)$ is the Legendre symbol.
\end{thm}

\begin{proof}
When $9\mid m$, (8.6) follows from Case 1 for Theorem 6.2,
so assume from now on that $9\nmid m$.
Since $3 \nmid n$, we may choose the signs of $m, n$
so that $m \e n \e 1 \pmod{3}$ when $3 \nmid m$,
and $n \e m/3 \e 1 \pmod{3}$  when $3 \mid m$.
It suffices to prove (8.6) for this choice of signs,
since if the sign of $m$ or $n$ is reversed, one
can take conjugates in (8.6).
Observe that the Legendre symbol $\jac{mn}{3}$ equals $1$
when $3 \nmid m$
and $\jac{mn/3}{3}$ equals $1$ when $3 \mid m$.

\noindent {\bf Case 1:} $d \e 2 \pmod{4}$

\noindent
In this case, $m_2:=m/2$ and $n_2:=n/2$ are relatively prime integers.
Set $x_1=y$ and $y_1=\f {y-x}3$. Then
\begin{equation}\label{8.7}
p=(1+3d)x_1^2-6x_1y_1+9y_1^2 
\end{equation}
and
\begin{equation}\label{8.8}
\f {2(1+3d)x_1-6y_1}{6dy_1} \e \f {x}{dy}\pmod p.
\end{equation}

First suppose that $3 \mid m$.
Using (8.1)-(8.3) and (8.5), we deduce that
\begin{align*}
\Big(\f {-6n-6m(1+2\omega)}{1+3d}\Big)_3
&=\Big(\f {m_2+n_2+m\omega}{1+3d}\Big)_3=
\Big(\f {m_2+n_2+m\omega}{(1+3d)n_2^2}\Big)_3
\Big(\f {m_2+n_2+m\omega}{n_2}\Big)_3\\
&=\Big(\f {(1+3d)n_2^2}{m_2+n_2 + m\omega}\Big)_3
\Big(\f {1+2\omega}{n_2}\Big)_3=\Big(\f {3+3m_2^2+n_2^2}
{m_2+n_2+m\omega}\Big)_3\\
&=\Big(\f {3}{m_2+n_2+m\omega}\Big)_3
=\omega^{\f  {2m}3}.
\end{align*}
Appealing to \cite[Theorem 5.1]{S} with the quadratic form (8.7), 
we obtain, using (8.8),
\begin{equation*}
\Big(\f {m+n\sqrt d}2\Big)^{\f {p-1}3}
   \e\f  12\Big(-1 +  \f {x}{dy}\sqrt d\Big)\pmod p,
\end{equation*}
 as desired.

Now assume that $3\nmid m$.
Observe that $d\e dn^2=m^2+4\e 2\pmod 3$. 
Using (8.1)-(8.3) and (8.5), we deduce that
\begin{align*}
&\Big(\f {-6n-6m(1+2\omega)}{1+3d}\Big)_3\\
&=\Big(\f {m_2+n_2+m\omega}{1+3d}\Big)_3
=\Big(\f{-\omega^2}{1+3d}\Big)_3\Big(\f{m+(m_2-n_2)\omega}{1+3d}\Big)_3\\
&=\Big(\f{\omega}{-1-3d}\Big)_3^2\Big(\f{m+(m_2-n_2)\omega}{1+3d}\Big)_3
=\omega^{-2d}
\Big(\f{m+(m_2-n_2)\omega}{n_2}\Big)_3
\Big(\f{m+(m_2-n_2)\omega}{(1+3d)n_2^2}\Big)_3\\
&=\omega^d\Big(\f{2+\omega}{n_2}\Big)_3
\Big(\f {(1+3d)n_2^2}{m+(m_2-n_2)\omega}\Big)_3
=\omega^2\Big(\f{-\omega^2(1-\omega)}{n_2}\Big)_3
\Big(\f{3+ 3m_2^2+n_2^2}{m+(m_2-n_2)\omega}\Big)_3\\
&=\omega^2\Big(\f{\omega}{n_2}\Big)_3\Big(\f{3}{m+(m_2-n_2)\omega}\Big)_3
=\omega^2\cdot \omega^{\f{n_2+1}3}
\Big(\f{3}{-m-(m_2-n_2)\omega}\Big)_3\\
&=\omega^2\cdot \omega^{\f{n_2+1}3}\cdot\omega^{-\f{m-n}3}
=\omega^{\f {4-m}3}.
\end{align*}
Appealing again to \cite[Theorem 5.1]{S} with the quadratic form (8.7),
we obtain (8.6).

\noindent {\bf Case 2:} $d \e 1 \pmod{4}$

\noindent
In this case, $m$ and $n$ have the same parity.
Set $x_1=\f {x-y}3$ and $y_1=\f {2x+4y}3$. Then
\begin{equation}\label{8.9}
p=(3d+4)x_1^2-(3d-2)x_1y_1+\f {3d+1}4y_1^2. 
\end{equation}

First suppose that $3 \mid m$.
Write $n_2 = n/(n,2)$ and $m_2 = m/(n,2)$.
Using (8.1)-(8.3) and also (8.5) with $n$ replaced by $-2n$, we deduce that
\begin{align*}&\Big(\f{-(3d-2)n-3m(1+2\omega)}{3d+4}\Big)_3
\\&=\Big(\f{6n-3m(1+2\omega)}{3d+4}\Big)_3
=\Big(\f{m-2n+2m\omega}{3d+4}\Big)_3
\\&=\Big(\f{m_2-2n_2+2m_2\omega}{3d+4}\Big)_3
=\Big(\f{m_2-2n_2+2m_2\omega}{n_2}\Big)_3
\Big(\f{m_2-2n_2+2m_2\omega}{(3d+4)n_2^2}\Big)_3
\\&=\Big(\f{1+2\omega}{n_2}\Big)_3\Big(\f{(3d+4)n_2^2}
{m_2-2n_2+2m_2\omega}\Big)_3
	=\Big(\f{3m_2^2+4n_2^2+12/(n,2)^2}
{m_2-2n_2+2m_2\omega}\Big)_3
	\\&=\Big(\f{12/(n,2)^2}
{m_2-2n_2+2m_2\omega}\Big)_3
=\Big(\f{3}
	{m_2-2n_2+2m_2\omega}\Big)_3\Big(\f{4/(n,2)^2}
{m_2-2n_2+2m_2\omega}\Big)_3
\\&=\Big(\f{3}
{m_2-2n_2+2m_2\omega}\Big)_3
\Big(\f
	{m_2-2n_2+2m_2\omega}{2/(2,n)}\Big)_3^2
=\Big(\f{3}
{m_2-2n_2+2m_2\omega}\Big)_3
=\omega^{\f {2m}3} = \omega^2.
\end{align*}
Appealing to \cite[Theorem 5.1]{S} with the quadratic form (8.9),
we obtain (8.6) in the case $3\mid m$.

Next suppose that $3 \nmid m$.
Note that $d=(m^2+4)/n^2 \e 2 \pmod{3}$.
From (8.1)--(8.3),
\begin{align*}
&\Big(\f{-(3d-2)n-3m(1+2\omega)}{3d+4}\Big)_3
\\&=\Big(\f{6n-3m(1+2\omega)}{3d+4}\Big)_3
=\Big(\f{m-2n+2m\omega}{3d+4}\Big)_3
\\&=\Big(\f{-\omega^2}{3d+4}\Big)_3\Big(\f{-\omega(m-2n+2m\omega)}{3d+4}\Big)_3
=\Big(\f{\omega^2}{3d+4}\Big)_3\Big(\f{2m+(m+2n)\omega}{3d+4}\Big)_3
\\&=\Big(\f{\omega}{-4-3d}\Big)_3^2
\Big(\f{2m_2+(m_2+2n_2)\omega}{3d+4}\Big)_3
\\&=\omega^{\f{1-4-3d}3}
\Big(\f{2m_2+(m_2+2n_2)\omega}{n_2}\Big)_3
\Big(\f{2m_2+(m_2+2n_2)\omega}{(3d+4)n_2^2}\Big)_3\\&
=\omega^{-d-1}\Big(\f{2+\omega}{n_2}\Big)_3 
\Big(\f{(3d+4)n_2^2}{2m_2+(m_2+2n_2)\omega}\Big)_3\\&
=\Big(\f{-\omega^2(1-\omega)}{n_2}\Big)_3
\Big(\f{3m_2^2+4n_2^2+12/(2,n)^2}{2m_2+(m_2+2n_2)\omega}\Big)_3
=\Big(\f{\omega}{n_2}\Big)_3
\Big(\f{12/(n,2)^2}{2m_2+(m_2+2n_2)\omega}\Big)_3
\\&=\Big(\f{\omega}{n_2}\Big)_3
\Big(\f{3}{2m_2+(m_2+2n_2)\omega}\Big)_3
\Big(\f{2/(n,2)}{2m_2+(m_2+2n_2)\omega}\Big)_3^2
\\&=\Big(\f{\omega}{n_2}\Big)_3
\Big(\f{3}{2m_2+(m_2+2n_2)\omega}\Big)_3
\Big(\f{2m_2+(m_2+2n_2)\omega}{2/(n,2)}\Big)_3^2
\\&=\Big(\f{\omega}{n_2}\Big)_3
\Big(\f{3}{2m_2+(m_2+2n_2)\omega}\Big)_3
	\Big(\f{\omega}{2/(n,2)}\Big)_3^2=\omega^{\f{4-m}3}.
\end{align*}
Appealing once again to \cite[Theorem 5.1]{S} with the quadratic form (8.9),
we obtain (8.6) in the case $3 \nmid m$, which completes the proof.
\end{proof}

From now on let $m^2-dn^2=4$, so that $(m-2)(m+2) = dn^2$. 
Always choose a sign of $m$ so that $\ord_3(m-2) \ge \ord_3 n$.
Set $m_1=\f{m-2}{(m-2,n)}$ and $n_1=\f n{(m-2,n)}$.
Then $(m_1,n_1)=1$ and $3\nmid n_1$. 
Fix the sign of $n$ so that $n_1 \e 1 \pmod{3}$.  
For brevity, define
$\a=\ord_3 (m-2)$, $\b = \ord_3 n$, and $\c = \ord_3 d$.
From the formula for $dn^2$, we have
$\a = \c + 2\b$ when $\b > 0$. 
If $3 \nmid m_1$ and $\b>0$, then $\b=\a=\c+2\b$, which is impossible.  Thus
\begin{equation} \label{8.10}
3 \nmid m_1 \Longleftrightarrow 3 \nmid (m-2)n.
\end{equation}
Let $m_4 = 4m_1/(m-2,n)$, which is an integer since
$m_4 = dn_1^2 - m_1^2$.   Let $m_0$
denote the odd part of $m_4$.
We have $m_0 \mid m_1$,  so $m_0$ is relatively prime with $n_1$.
Also $m_0$ divides $m_4 (m+2) =4dn_1^2$.   Thus
\begin{equation} \label{8.11}
m_0 \mid d.
\end{equation}
Consequently $(m_4,3d+1)=1$ when $2\mid d$ and $(m_4,3d+4)=1$ when $2\nmid d$.

We claim that 
\begin{equation}\label{8.12}
9 \mid m_1 \Longleftrightarrow m \e 2 \pmod{27}.
\end{equation}
To see this, observe that (8.12) is equivalent to
\begin{equation}\label{8.13}
\a-\b \ge2  \Longleftrightarrow \a \ge 3.
\end{equation}
If $\b=0$, then $\a \le 1$, so both sides of (8.13) are false.
If $\b>0$, then
(8.13) is equivalent to
\begin{equation}\label{8.14}
\c+\b \ge2  \Longleftrightarrow \c+2\b \ge 3.
\end{equation}
Since $\c=0$ or $\c=1$, (8.14) holds, so (8.12) is proved.

In view of (8.12), the proof in Case 2 of Theorem 6.2
shows that if either $9 \mid m$ or $m \e 2 \pmod{27}$,
then $\f {m+n\sqrt d}2$ is a cubic residue of the primes  $p=x^2+3dy^2$.
The converse is a consequence of the following technical theorem.

\begin{thm}\label{Thm 8.2}
Suppose that $m^2-dn^2=4$, $9 \nmid m$, $27 \nmid (m-2)$,
and $\ord_3(m-2) \ge \ord_3 n$.
Then modulo $p=x^2 + 3dy^2$, we have
\begin{equation}\label{8.15}
\Big(\f {m+n\sqrt d}2\Big)^{\f {p-1}3}\e 
\begin{cases} 
	\f  12\big(-1+ (\f {(m-2)n/3}3)\f {x}{dy}\sqrt d\big)
&\hbox{if $m\e 5,8 \pmod 9$,} \\
\f  12\big(-1- (\f {(m-2)n/27}3)\f
{x}{dy}\sqrt d\big)
	&\hbox{if $m\e  11,20\pmod {27}$,}\\
\f  12\big(-1+ (\f{mn/3}3)\f {x}{dy}\sqrt d\big)
&\hbox{if $m\e 0\pmod 3$,}\\
\f  12\big(-1+ (\f{(m+2)n/3}3)\f {x}{dy}\sqrt d\big)
&\hbox{if $m\e 1\pmod 3$,}
\end{cases}
\end{equation}
 where $\big(\f{\cdot}3\big)$ is the Legendre symbol.
\end{thm}

\begin{proof}
Note that $9 \nmid n$, otherwise $27$ would divide $m-2$.
When $3 \mid (m-2)$, we have $3\parallel m_1$ by (8.10) and (8.12).

\noindent {\bf Case 1:} $2 \mid d$.

\noindent
In this case $(m,n)=2$.
Suppose first that $3 \mid (m-2)$, so that either
$m \e 5,8 \pmod 9$ or $m \e 11,20 \pmod{27}$.
Using (8.1)--(8.3) and (8.5), we see that
\begin{align*}
&\Big(\f{-6n_1+6m_1(1+2\omega)}{3d+1}\Big)_3
\\&=\Big(\f{m_1-n_1+2m_1\omega}{3d+1}\Big)_3
=\Big(\f{m_1-n_1+2m_1\omega}{n_1}\Big)_3
\Big(\f{m_1-n_1+2m_1\omega}{(3d+1)n_1^2}\Big)_3
 \\&=\Big(\f{1+2\omega}{3d+1}\Big)_3
 \Big(\f{(3d+1)n_1^2}{m_1-n_1+2m_1\omega}\Big)_3
 =\Big(\f{3m_1^2+n_1^2+3m_4}{m_1-n_1+2m_1\omega}\Big)_3
 \\&=\Big(\f{3m_4}{m_1-n_1+2m_1\omega}\Big)_3
=\Big(\f{3}{m_1-n_1+2m_1\omega}\Big)_3\Big(\f{m_4}{m_1-n_1+2m_1\omega}\Big)_3.
 \end{align*}

Assume that $m \e 5,8 \pmod 9$.
Then $3\ \Vert\ m_1$, $3\nmid n$ and $3\ \Vert\ \f{m-2}{(m-2,n)^2}$.
From the above we deduce that
\begin{align*}
&\Big(\f{-6n_1+6m_1(1+2\omega)}{3d+1}\Big)_3
\\&=\Big(\f{3^2}{m_1-n_1+2m_1\omega}\Big)_3
\Big(\f{m_4/3}{m_1-n_1+2m_1\omega}\Big)_3
\\&=\omega^{\f{2m_1}3}
\Big(\f {m_1-n_1+2m_1\omega}{m_4/3}\Big)_3
\\&=\omega^{\f{2m_1}3}
\Big(\f {-n_1}{m_4/3}\Big)_3
=\omega^{\f{2m_1}3}=\omega^{2(\f{(m-2)n/3}3)},
\end{align*}
where the last equality follows because $(m-2,n)^2 \e 1 \pmod 3$.

Applying \cite[Thm. 5.5]{S} with the quadratic form (8.7),
and using (8.8), we obtain the first line in (8.15).

Next assume
$m\e 11,20\pmod {27}$. 
Then $9\ \Vert\ m-2$. By (8.10) and (8.12),
$3\ \Vert\ m_1$. Hence $3\ \Vert\  n$ and $3\nmid m_4$. This time
\begin{align*}
&\Big(\f{-6n_1+6m_1(1+2\omega)}{3d+1}\Big)_3\\&
=\Big(\f{3}{m_1-n_1+2m_1\omega}\Big)_3
\Big(\f{m_4}{m_1-n_1+2m_1\omega}\Big)_3
\\&=\omega^{\f{m_1}3}
\Big(\f {m_1-n_1+2m_1\omega}{m_4}\Big)_3
=\omega^{\f{m_1}3}
\Big(\f {-n_1}{m_4}\Big)_3
=\omega^{\f{m_1}3}=\omega^{(\f{(m-2)n/27}3)},
\end{align*}
where the last equality follows because $(m-2,n)^2 \e 0\pmod 9$.

Applying \cite[Thm. 5.5]{S} with the quadratic form (8.7),
and using (8.8),we obtain the second line in (8.15).

Finally assume that  $m\e 0,1\pmod 3$. 
We have $9 \nmid m$, $3\nmid m-2$ and $3 \nmid n$ 
since $\ord_3(m-2)\ge \ord_3 n$. 
Hence $3\nmid m_1n_1$ and so
$4(2-m)\e m_4=m_1^2-dn_1^2\e 1-d\pmod 3$. This implies $d\e m-1\pmod 3$.

First assume $m_1\e n_1\pmod 3$.
Then
\begin{align*}
&\Big(\f{-6n_1+6m_1(1+2\omega)}{3d+1}\Big)_3
\\&=\Big(\f{m_1-n_1+2m_1\omega}{3d+1}\Big)_3
=\Big(\f{\omega}{3d+1}\Big)_3\Big(\f{m_1+n_1-(m_1-n_1)\omega}{3d+1}\Big)_3
\\&=\omega^{-d}\Big(\f{m_1+n_1-(m_1-n_1)\omega}{n_1}\Big)_3
\Big(\f{m_1+n_1-(m_1-n_1)\omega}{(3d+1)n_1^2}\Big)_3
\\&=\omega^{-d}\Big(\f{1-\omega}{n_1}\Big)_3
\Big(\f{(3d+1)n_1^2}{m_1+n_1-(m_1-n_1)\omega}\Big)_3
\\&=\omega^{-d}\Big(\f{1-\omega}{n_1}\Big)_3
\Big(\f{3m_1^2+n_1^2+3m_4}{m_1+n_1-(m_1-n_1)\omega}\Big)_3
\\&=\omega^{-d}\Big(\f{1-\omega}{-n_1}\Big)_3
\Big(\f{3m_4}{m_1+n_1-(m_1-n_1)\omega}\Big)_3
	\\&=\omega^{-d}\cdot \omega^{\f {2(1-n_1)}3}
\Big(\f{3}{m_1+n_1-(m_1-n_1)\omega}\Big)_3
\Big(\f{m_4}{m_1+n_1-(m_1-n_1)\omega}\Big)_3
	\\&= \omega^{\f {2(1-n_1)}3-d}
\Big(\f{3}{m_1+n_1-(m_1-n_1)\omega}\Big)_3
\Big(\f{m_1+n_1-(m_1-n_1)\omega}{m_4}\Big)_3
\\&= \omega^{\f {2(1-n_1)}3 -d}\cdot 
\omega^{\f {2(n_1-m_1)}3}
\Big(\f{1+\omega}{m_4}\Big)_3
= \omega^{\f {2(1-m_1)}3-d}
\Big(\f{-\omega^2}{m_4}\Big)_3\\&
= \omega^{\f {m_1-1}3+(1-m)}\Big(\f{\omega}{-m_4^2}\Big)_3
=\omega^{\f{1-(m+1)(\f {m-2}3)}3},
\end{align*}
where the last equality follows (after a tedious calculation)
using the congruence $m-2 \e (m-2,n) \pmod 3$.

If on the other hand  $m_1\e -n_1\pmod 3$, then by (8.4),
\begin{equation*}
\Big(\f{-6n_1+6m_1(1+2\omega)}{3d+1}\Big)_3=
\Big(\f{-6(-n_1)+6m_1(1+2\omega)}{3d+1}\Big)_3^{-1}=
\omega^{-\f{1-(m+1)(\f{m-2}3)}3}.
\end{equation*}
Since $m_1\e (\f{(m-2)n}3)n_1\pmod 3$, we have in either case
\begin{equation*}
\Big(\f{-6n_1+6m_1(1+2\omega)}{3d+1}\Big)_3=
\omega^{(\f{(m-2)n}3)\f{1-(m+1)(\f{m-2}3)}3}
=\omega^{(\f{n}3)\f{(\f{m-2}3)-(m+1)}3}.
\end{equation*}
The rightmost member  equals
$\omega^{-(\f{mn/3}3)}$ or  $\omega^{-(\f{(m+2)n/3}3)}$
according as $m$ is congruent to $0$ or $1$ mod $3$.
Applying \cite[Theorem 5.5]{S} with the quadratic form (8.7), 
and using (8.8), we obtain the last two lines of (8.15).

\noindent {\bf Case 2:}  $2 \nmid d$.

\noindent
Setting  $x_1=y$ and $y_1=\f {x+2y}3$, we have
\begin{equation}\label{8.16}
p=(3d+4)x_1^2-12x_1y_1+9y_1^2,
\end{equation}
and setting $x_2=\f{x-y}3$ and $y_2=-\f{2x+4y}3$, we get
\begin{equation}\label{8.17}
p=(3d+4)x_2^2+(3d-2)x_2y_2+\f{3d+1}4y_2^2\quad\t{for}\quad d\e 1\pmod 4.
\end{equation}
Note that
\begin{equation}\label{8.18}
\Big(\f{-12n_1+6m_1(1+2\omega)}{3d+4}\Big)_3=
\Big(\f{(3d-2)n_1+3m_1(1+2\omega)}{3d+4}\Big)_3.
\end{equation}
It is easy to check that
\begin{equation}\label{8.19}
\f {2(3d+4)x_1-12y_1}{6dy_1}=\f{2(3d+4)x_2+(3d-2)y_2}{3dy_2}
\e -\f {x}{dy}\pmod p
\end{equation}
and
\begin{equation}\label{8.20}
(m_1-2n_1+2m_1\omega)(m_1-2n_1+2m_1\omega^2) =3m_1^2+4n_1^2
=(3d+4)n_1^2 -3m_4.
\end{equation}
Formulas (8.17)--(8.19) will be needed later when
applying \cite[Theorem 5.5]{S} for the case $d \e 1 \pmod 4$.

Suppose first that $3 \mid (m-2)$, so that either
$m \e 5,8 \pmod 9$ or $m \e 11,20 \pmod{27}$.
Using (8.1)--(8.3) and (8.20), we see that
\begin{align*}
&\Big(\f{-12n_1+6m_1(1+2\omega)}{3d+4}\Big)_3
\\&=\Big(\f{m_1-2n_1+2m_1\omega}{3d+4}\Big)_3=\Big(\f{m_1-2n_1+2m_1\omega}{n_1}\Big)_3
\Big(\f{m_1-2n_1+2m_1\omega}{(3d+4)n_1^2}\Big)_3
\\&=\Big(\f{1+2\omega}{n_1}\Big)_3
\Big(\f{(3d+4)n_1^2}{m_1-2n_1+2m_1\omega}\Big)_3
\\&=\Big(\f{3m_1^2+4n_1^2+3m_4}{m_1-2n_1+2m_1\omega}\Big)_3
=\Big( \f{3^2\cdot m_4/3}{m_1-2n_1+2m_1\omega}\Big)_3.
\end{align*}

Assume that $m\e 5,8\pmod 9$. 
Then $3\ \Vert\ m_1$, $3\nmid n$ and $3\ \Vert\ m_4$.  We have
\begin{align*}&\Big( \f{m_4/3}{m_1-2n_1+2m_1\omega}\Big)_3
	=\Big( \f{m_1-2n_1+2m_1\omega}{m_4/3}\Big)_3\\&
=\Big( \f{m_1-2n_1+2m_1\omega}{m_0/3}\Big)_3
=\Big( \f{-2n_1}{m_0/3}\Big)_3=1.
\end{align*}

Hence,
\begin{align*}
&\Big(\f{-12n_1+6m_1(1+2\omega)}{3d+4}\Big)_3
=\Big( \f{3}{-m_1+2n_1-2m_1\omega}\Big)_3^2\\&
=\omega^{\f{m_1}3}
=\omega^{(\f{m_1n_1/3}3)}=\omega^{(\f{(m-2)n/3}3)}.
\end{align*}

Applying \cite[Thm. 5.5]{S} with the quadratic forms (8.16),(8.17)
and using (8.19), we obtain the first line in (8.15).

Next assume
$m\e 11,20\pmod {27}$.
Then $9\ \Vert\ m-2$. By (8.10) and (8.12),
$3\ \Vert\ m_1$. Hence $3\ \Vert\  n$ and $3\nmid m_4$. We have

\begin{align*}&\Big( \f{m_4}{m_1-2n_1+2m_1\omega}\Big)_3
        =\Big( \f{m_1-2n_1+2m_1\omega}{m_4}\Big)_3\\&
=\Big( \f{m_1-2n_1+2m_1\omega}{m_0}\Big)_3
=\Big( \f{-2n_1}{m_0}\Big)_3=1.
\end{align*}

Hence,
\begin{align*}
&\Big(\f{-12n_1+6m_1(1+2\omega)}{3d+4}\Big)_3
=\Big( \f{3}{-m_1+2n_1-2m_1\omega}\Big)_3\\&
=\omega^{\f{2m_1}3}
=\omega^{(\f{2m_1n_1/3}3)}=\omega^{(\f{2(m-2)n/27}3)}.
\end{align*}

Applying \cite[Thm. 5.5]{S} with the quadratic forms (8.16),(8.17)
and using (8.19), we obtain the second line in (8.15).

Finally assume that  $m\e 0,1\pmod 3$.
As in Case 1, we have $9 \nmid m$, $3\nmid m-2$ 
and $3\nmid n$.
Hence $3\nmid m_1n_1$ and so
$4(2-m)\e m_4=m_1^2-dn_1^2\e 1-d\pmod 3$. This implies $d\e m-1\pmod 3$.

First assume $m_1\e n_1\pmod 3$.
Then
\begin{align*}
&\Big(\f{-12n_1+6m_1(1+2\omega)}{3d+4}\Big)_3
\\&=\Big(\f{m_1-2n_1+2m_1\omega}{3d+4}\Big)_3
=\Big(\f{-\omega^2}{3d+4}\Big)_3\Big(\f{2m_1+(m_1+2n_1)\omega}{3d+4}\Big)_3
\\&=\Big(\f{\omega}{-4-3d}\Big)_3^2\Big(\f{2m_1+(m_1+2n_1)\omega}{n_1}\Big)_3
\Big(\f{2m_1+(m_1+2n_1)\omega}{(3d+4)n_1^2}\Big)_3
\\&=\omega^{\f{2(1-4-3d)}3}\Big(\f{2+\omega}{n_1}\Big)_3\Big(\f{(3d+4)n_1^2}
{2m_1+(m_1+2n_1)\omega}\Big)_3
\\&=\omega^{d+1}\Big(\f{-\omega^2(1-\omega)}{n_1}\Big)_3
\Big(\f{3m_1^2+4n_1^2+3m_4}
{2m_1+(m_1+2n_1)\omega}\Big)_3
\\&=\omega^{m}\Big(\f{\omega^2(1-\omega)}{-n_1}\Big)_3
\Big(\f{3m_4}{2m_1+(m_1+2n_1)\omega}\Big)_3
\\&=\omega^{m}\cdot\omega^{\f{4(1-n_1)}3}
\Big(\f 3{2m_1+(m_1+2n_1)\omega}\Big)_3
\Big(\f{2m_1+(m_1+2n_1)\omega}{m_4}\Big)_3
\\&=\omega^{m}\cdot\omega^{\f{4(1-n_1)}3}\cdot \omega^{\f{2(m_1+2n_1)}3} 
\Big(\f{\omega}{m_4}\Big)_3
\end{align*}
because $m_0 \mid m_1$ and $m_4$ equals $m_0$ times a power of $2$.
Therefore
\begin{equation*}
\Big(\f{-12n_1+6m_1(1+2\omega)}{3d+4}\Big)_3=
\omega^{m+\f{4-4n_1+2m_1+4n_1+1-(\f{m-2}3)m_4}3} 
=\omega^{\f{(m+1)(\f{m-2}3)-1}3},
\end{equation*}
where the last equality follows (after a tedious calculation)
using the congruence $m-2 \e (m-2,n) \pmod 3$.

If on the other hand  $m_1\e -n_1\pmod 3$, then by (8.4),
\begin{equation*}
\Big(\f{-12n_1+6m_1(1+2\omega)}{3d+4}\Big)_3=
\Big(\f{-12(-n_1)+6m_1(1+2\omega)}{3d+1}\Big)_3^{-1}=
\omega^{-\f{(m+1)(\f{m-2}3)-1}3}.
\end{equation*}
Since $m_1\e (\f{(m-2)n}3)n_1\pmod 3$, we have in either case
\begin{equation*}
\Big(\f{-12n_1+6m_1(1+2\omega)}{3d+1}\Big)_3=
\omega^{(\f{(m-2)n}3)\f{(m+1)(\f{m-2}3)-1}3}
	=\omega^{(\f{n}3)\f{m+1 -(\f{m-2}3)}3}.
\end{equation*}
The rightmost member  equals
$\omega^{(\f{mn/3}3)}$ or  $\omega^{(\f{(m+2)n/3}3)}$
according as $m$ is congruent to $0$ or $1$ mod $3$.
Applying \cite[Theorem 5.5]{S} with the quadratic forms (8.16),(8.17)
and using (8.19), we obtain the last two lines of (8.15).
\end{proof}

For real $b$, $c$, let $\{U_k(b,c)\}$ be the Lucas sequence defined by
\begin{equation*}
U_0=0,\ U_1=1,\quad U_{k+1}=bU_k-cU_{k-1}\quad (k=1,2,3,\ldots).
\end{equation*}
It is well known that for $b^2-4c\ne 0$ and $k\ge 0$,
\begin{equation*}
U_k(b,c)=\f 1{\sqrt{b^2-4c}}\Big(
\Big(\f{b+\sqrt{b^2-4c}}2\Big)^k-\Big(\f{b-\sqrt{b^2-4c}}2\Big)^k\Big).
\end{equation*}
For squarefree $d>1$, nonzero integers $m$, $n$, and $\varepsilon = \pm 1$,
write  $m^2-dn^2=4\varepsilon$.  Then
\begin{equation*}
U_k(m,\varepsilon)=\f 1{n\sqrt{d}}\Big(
\Big(\f{m+n\sqrt{d}}2\Big)^k-\Big(\f{m-n\sqrt{d}}2\Big)^k\Big).
\end{equation*}

The two corollaries below evaluate 
the Lucas numbers $U_{\f{p-1}3}(m,\varepsilon) \pmod p$
for primes $p = x^2 +3dy^2$ with $p \nmid n$.   
When $3 \mid y$, it follows from
Theorem 6.2  that $\f{m+n\sqrt d}2$ is a cubic residue mod $p$,
so that $U_{\f{p-1}3}(m,\varepsilon) \e 0 \pmod p$.
Thus we assume that $3 \nmid y$.  Fix the signs
of $x$, $y$  so that $x \e y \pmod 3$.
From Theorems 6.1, 8.1 and 8.2,  we deduce:

\begin{cor}\label{Cor 8.3}
 
Suppose that $m^2-dn^2=-4$. Then
\begin{equation*}
U_{\f{p-1}3}(m,-1) \e
	\begin{cases} 0\pmod p&\hbox{if $m\e 0,\pm 4\pmod{9}$,}
\\ \big(\f{mn/3}3\big) \f {x}{dny}\pmod p
	&\hbox{if $m\e \pm 3\pmod 9$,}
\\-\big(\f{mn}3\big)\f {x}{dny}\pmod p
	&\hbox{if $m\e \pm 1 \pmod 9$,}
\\\big(\f{mn}3\big)\f {x}{dny}\pmod p
	&\hbox{if $m\e \pm 2\pmod 9$.}
\end{cases}
\end{equation*}
\end{cor}
\begin{cor}\label{Cor 8.4}

Suppose that $m^2-dn^2=4$ with a sign of $m$ chosen
such that $\ord_3 (m-2) \ge \ord_3 n$. 
Then
\begin{equation*}
U_{\f{p-1}3}(m,1) \e
\begin{cases} 0\pmod p&\hbox{if $m\e 2\pmod{27}$,}
\\ \big(\f{(m-2)n/3}3\big) \f {x}{dny}\pmod p
	&\hbox{if $m\e 5,8\pmod 9$,}
\\-\big(\f{(m-2)n/27}3\big)\f {x}{dny}\pmod p
	&\hbox{if $m\e 11,20\pmod {27}$,}
\\\big(\f{mn/3}3\big)\f {x}{dny}\pmod p
	&\hbox{if $m\e 0 \pmod 3$,}
\\\big(\f{(m+2)n/3}3\big)\f {x}{dny}\pmod p
	&\hbox{if $m\e 1\pmod 3$.}
\end{cases}
\end{equation*}
\end{cor}

\begin{exa}
	Take $m=2t$, so that $\{U_k(m,1)\}$ is a sequence
of Chebyshev polynomials of the second kind in the argument $t$.
First suppose that $t$ is an integer multiple of $3$.
Choose $s \e 1 \pmod 3$ such that $(t^2-1)/s^2$ is 
squarefree.  Then by Corollary 8.4 with $n=2s$ and $d=(t^2-1)/s^2$, we have
$U_{\f{p-1}3}(m,1) \e \big(\f{t/3}3\big)\f {xs}{2y(t^2-1)}\pmod p$.
For example, with $t=21$, $s=-2$, $d=110$, $x=y=1$, and $p=331$, 
we have
$U_{110}(m,1)  \e 249 \pmod {331}$.
Next suppose that $t \e 2 \pmod 3$. Note that $9 \nmid (t^2-1)$,
otherwise $3 \mid n$, contradicting $\ord_3 (m-2) = 0$.
With $s$, $d$, and $n$ as above, it follows from Corollary 8.4 that
$U_{\f{p-1}3}(m,1) \e \big(\f{(t+1)/3}3\big)\f {xs}{2y(t^2-1)}\pmod p$.
\end{exa}


\begin{thebibliography}{99}
\bibitem{CG}
S. Cort\'es-G\'omez,
Higher composition laws, integral trace forms and 
an alternative proof for the Scholz reflection principle, (2018) \\
\url{https://repositorio.uniandes.edu.co/server/api/core/
bitstreams/dc1e4d53-2886-49ee-a625-f03665bf831b/content}

\bibitem{Chil}
N. Childress, \emph{Class Field Theory},
Springer, New York,  2009.

\bibitem{Cho}
B. Cho, Primes of the form $p=x^2 + n y^2$
with conditions $x \equiv 1 \pmod{N}$, $y \equiv 0 \pmod{N}$,
J. Number Theory 130 (2010), 852--861.

\bibitem{Cox}
D. Cox, \emph{Primes of the form $p=x^2 + n y^2$},
3rd ed., AMS Chelsea, Providence, 2022.

\bibitem{EVV}
R. Evans, M. Van Veen, Mathematica notebook, (2024) \\
\url{https://math.ucsd.edu/~revans/CubicData.nb}

\bibitem{H} 
H. Hasse,
Arithmetische Theorie der kubischen Zahlk{\"o}rper auf
klassenk{\"o}rpertheoretischer Grundlage, Math. Z. 31 (1930),  565--582.

\bibitem{Herz}
C. Herz, Construction of class fields. 
In: Seminar on Complex Multiplication. Lecture Notes in Mathematics, vol 21
(1957, revised Nov. 1965)
pp. 71--91, Springer, Berlin, Heidelberg. 

\bibitem{HC}
A. Hogue, K. Chakraborty, Divisibility of class numbers of certain
families of quadratic fields, J. Ramanujan Math. Soc. 34 (2019),
281--289.

\bibitem{I}
A.Ito, On the 3-divisibility of class numbers of pairs of
quadratic fields with splitting conditions,
Functiones et Approx. 60(1) (2019), 61--76.


\bibitem{KM}
S. Krishnamoorthy, R. Muneeswaran,
Divisibility of the class number of the imaginary quadratic
fields $\Q(\sqrt{1-2m^k})$, (2024) \\
\url{https://arxiv.org/pdf/2111.04387v4.pdf}

\bibitem{N}
W. Narkiewicz, \emph{Elementary and Analytic Theory of Algebraic Numbers},
3rd ed., Springer, Berlin, 2004.

\bibitem{S}
Z.-H. Sun, Cubic residues and binary quadratic forms,
J. Number Theory 124 (2007), 62--104.

\bibitem{V}
T. Vaughn,
The discriminant of a quadratic extension of an algebraic field,
Math. Comp. 40 (1983),  685--707.

\bibitem{XC}
J. Xie, K. Chao, A note on 3-divisibility of class number of
quadratic field, Chin. Ann. Math. Ser. B 43(2) (2022), 307--318.

\end{thebibliography}
\end{document}